\documentclass[12pt,reqno]{amsart}
\usepackage{amssymb,latexsym,graphicx,amscd,upgreek}
\usepackage{lscape}
\usepackage[all]{xy}
\usepackage{color}
\usepackage{epic,eepic}
\usepackage{xspace}
\usepackage{mathrsfs}
\usepackage{setspace}
\usepackage{cases}
\usepackage{bbm}
\usepackage{hyperref}
\hypersetup{
    colorlinks,
    linkcolor={blue},
    citecolor={blue},
    urlcolor={blue!80!black}
}

\newtheorem{Thm}{Theorem}[section]
\newtheorem{Lem}[Thm]{Lemma}
\newtheorem{Cor}[Thm]{Corollary}
\newtheorem{Prop}[Thm]{Proposition}

\theoremstyle{definition}

\newcommand{\Z}{\mathbb{Z}}

\newcommand{\N}{\mathbb{N}}

\newcommand{\df}{\colon}

\newcommand{\cC}{{\mathcal C}}

\newcommand{\cE}{{\mathcal E}}

\newcommand{\cO}{{\mathcal O}}

\newcommand{\cU}{{\mathcal U}}

\newcommand{\cZ}{{\mathcal Z}}

\newcommand{\bd}{\mathbf{d}}

\newcommand{\rk}{\operatorname{rk}}

\newcommand{\md}{\operatorname{mod}}

\newcommand{\rep}{\operatorname{rep}}

\newcommand{\ind}{\operatorname{ind}}

\newcommand{\add}{\operatorname{add}}

\newcommand{\Spec}{\operatorname{Spec}}

\newcommand{\gldim}{\operatorname{gl.dim}}
\newcommand{\pdim}{\operatorname{proj.dim}}
\newcommand{\idim}{\operatorname{inj.dim}}
\newcommand{\dm}{{\rm dim}\,}
\newcommand{\dimv}{\underline{\dim}}

\newcommand{\soc}{\operatorname{soc}}
\newcommand{\rad}{\operatorname{rad}}
\newcommand{\tp}{\operatorname{top}}

\newcommand{\supp}{\operatorname{supp}}

\newcommand{\zero}{\operatorname{null}}

\newcommand{\proj}{\operatorname{proj}}
\newcommand{\inj}{\operatorname{inj}}

\newcommand{\Hom}{\operatorname{Hom}}
\newcommand{\Ext}{\operatorname{Ext}}

\newcommand{\ext}{\operatorname{ext}}
\newcommand{\End}{\operatorname{End}}

\newcommand{\Aut}{\operatorname{Aut}}

\newcommand{\rigid}{\operatorname{rigid}}

\newcommand{\Ima}{\operatorname{Im}}
\newcommand{\Ker}{\operatorname{Ker}}
\newcommand{\Coker}{\operatorname{Cok}}

\newcommand{\bil}[1]{\langle #1\rangle}

\newcommand{\irr}{\operatorname{Irr}}

\newcommand{\bsm}{\begin{smallmatrix}}
\newcommand{\esm}{\end{smallmatrix}}

\newcommand{\bbsm}{\left[\begin{smallmatrix}}
\newcommand{\besm}{\end{smallmatrix}\right]}

\newcommand{\bbm}{\begin{matrix}}
\newcommand{\ebm}{\end{matrix}}

\newcommand{\GL}{\operatorname{GL}}

\newcommand{\calU}{\mathcal{U}}
\newcommand{\calZ}{\mathcal{Z}}

\DeclareMathOperator{\trace}{trace}

\begin{document}


\title{Generically $\tau$-regular irreducible components of module varieties}

\author{Grzegorz Bobi\'nski}
\address{Grzegorz Bobi\'nski\newline
Faculty of Mathematics and Computer Science\newline
Nicolaus Copernicus University\newline
ul. Chopina 12/18\newline
87-100 Toru\'n\newline
Poland}
\email{gregbob@mat.umk.pl}

\author{Jan Schr\"oer}
\address{Jan Schr\"oer\newline
Mathematisches Institut\newline
Universit\"at Bonn\newline
Endenicher Allee 60\newline
53115 Bonn\newline
Germany}
\email{schroer@math.uni-bonn.de}


\begin{abstract}
In the representation theory of finite-dimensional algebras,
the study of projective presentations of maximal rank is closely related
to the study of generically $\tau$-regular irreducible components of
varieties of modules over such algebras.
We show that a module is $\tau$-regular if and only if
its minimal projective presentation is of maximal rank.
This is a refinement of a theorem by Plamondon.
We prove that generic extensions of generically $\tau$-regular components by simple projective modules are
again generically $\tau$-regular.
This leads to the classification of all generically $\tau$-regular components for triangular algebras.
We also show that an algebra is hereditary if and only if
all irreducible components of its varieties of modules are generically
$\tau$-regular.
Finally, we discuss when the set of generically $\tau$-regular
components coincides with the set of generically $\tau^-$-regular
components.
\end{abstract}

\maketitle

\setcounter{tocdepth}{1}
\numberwithin{equation}{section}
\tableofcontents

\parskip2mm


\section{Introduction and main results}


Throughout, $K$ denotes an algebraically closed field, and
$A$ is always a finite-dimensional associative $K$-algebra.

\subsection{Generic representation theory}
We outline two concepts of
\emph{generic representation theory}.

For a dimension vector $\bd$ let
$\md(A,\bd)$ be the affine variety of $A$-modules
with dimension vector $\bd$.
The corresponding product $G_\bd$ of general linear groups
acts on $\md(A,\bd)$ by conjugation.
The orbits of this action correspond to the isomorphism classes of $A$-modules with dimension vector $\bd$.
For $M \in \md(A,\bd)$ we denote its orbit by $\cO_M$.
Let $\irr(A,\bd)$ be the set of irreducible components of $\md(A,\bd)$.

For $\calZ \in \irr(A,\bd)$ it is usually impossible to describe all modules
in $\calZ$. But one might be able to describe (or at least say something
about) the modules in a (small enough) dense open subset $\calU$ of $\calZ$.
Here one usually chooses a set of upper or lower semicontinuous maps
on $\calZ$ and takes $\calU$ as the subset where these maps assume their
generic value.

For example, if $\calZ$ contains a module $M$ with $\Ext_A^1(M,M) = 0$,
then $\calU = \cO_M$ is open and therefore $\calZ = \overline{\cO_M}$.

As another important class of examples, assume that
$A$ is a hereditary algebra, i.e.\ $\gldim(A) \le 1$.
Then $\md(A,\bd)$ is an affine space for all $\bd$.
In particular, there is just one irreducible component, namely $\md(A,\bd)$ itself.
An important result by Kac \cite[Proposition~1]{K82} implies that
there is a dense open subset $\calU$ of $\md(A,\bd)$ which consists of
modules of the form $M_1 \oplus \cdots \oplus M_t$ where
$\End_A(M_i) \cong K$ for all $1 \le i \le t$.

One could call this approach \emph{generic representation theory},
where the term \emph{generic} refers to (small enough) dense open subsets of
irreducible components.

There is a similar idea, but from a different angle.
Namely,
each $M \in \md(A)$ has a projective presentation
$$
P_1 \to P_0 \to M \to 0.
$$
So for each pair $(P_1,P_0)$ of projective modules we can look
at a (small enough) dense open subset $\calU$ of $\Hom_A(P_1,P_0)$ and
then try to study the modules $\Coker(f)$ where $f$ runs through $\calU$.
So one studies \emph{generic} projective presentations.
We refer to ground laying work by Derksen and Fei \cite{DF15} in this direction.

For $M \in \md(A,\bd)$ let
$$
\md_M(A,\bd) :=
\bigcup_{\substack{\calZ \in \irr(A,\bd)\\M \in \calZ}} \calZ,
$$
and define
\begin{align*}
c(M) &:= \dim \md_M(A,\bd) - \dim \cO_M,
\\
E(M) &:= \dim \Hom_A(M,\tau_A(M)).
\end{align*}
Here $\tau_A$ is the Auslander-Reiten translation for $A$.
Then $c(M) \le E(M)$ by Voigt's Lemma and the Auslander-Reiten
formula.

We call $M$ \emph{$\tau$-regular} if
$$
c(M) = E(M).
$$
(In this case, there is only one irreducible component which contains $M$.)
A detailed motivation for this terminology can be found in
Section~\ref{subsec:deftaureg}.
A component $\calZ \in \irr(A,\bd)$ is \emph{generically $\tau$-regular} if
for generic $M$ in $\calZ$ we have $c(M) = E(M)$.
(By upper semicontinuity, this is the case if and only if
$\calZ$ contains some $M$ with $c(M) = E(M)$.)
This class of irreducible components was first introduced in \cite{GLS12}.

For any pair $(P_1,P_0)$ of projective $A$-modules
let
\begin{align*}
r(P_1,P_0) &:= \max\{ \rk(f) \mid f \in \Hom_A(P_1,P_0) \},
\\
\calU &:=
\{  f \in \Hom_A(P_1,P_0) \mid \rk(f) = r(P_1,P_0) \},
\\
\calZ &:=
\overline{\{ M \in \md(A,\bd) \mid M \cong \Coker(f),\;
f \in \calU \}}
\end{align*}
where $\bd := \dimv(\Coker(f))$ for $f \in \calU$.

A beautiful result by
Plamondon \cite{P13} says that
$\calZ$
is a generically $\tau$-regular irreducible component
and that all of these components occur in this way.
Plamondon also shows that
the set $\irr^\tau(A)$ of generically $\tau$-regular components is (roughly speaking) parametrized by $\Z^n$, where
$n$ is the number of simple $A$-modules, up to isomorphism.
(A precise statement is in Section~\ref{subsec:plamondon}.)

Thus studying generically $\tau$-regular components corresponds to studying modules which have projective presentations of maximal rank.

For hereditary algebras $A$, the two concepts of generic representation theory outlined above coincide,
i.e.\ $\md(A,\bd)$ is a generically $\tau$-regular component for
all $\bd$.

Dually, one defines generically $\tau^-$-regular components.
We denote this set by $\irr^{\tau^-}(A)$.

\subsection{Main results}
The following theorem characterizes the modules whose
minimal
projective presentation is of maximal rank.

\begin{Thm}\label{thm:main0}
For $M \in \md(A)$ let
$$
P_1 \xrightarrow{f} P_0 \to M \to 0
$$
be a projective presentation.
Then
the following hold:
\begin{itemize}\itemsep2mm

\item[(i)]
If $\rk(f) = r(P_1,P_0)$, then $M$ is $\tau$-regular;

\item[(ii)]
If $M$ is $\tau$-regular and the presentation above is minimal,
then $\rk(f) = r(P_1,P_0)$.

\end{itemize}
\end{Thm}

As a consequence of
Plamondon's Theorem~\ref{thm:plamondon1},
Theorem~\ref{thm:main0}(i) holds
for
each \emph{generic} $f$ in $\Hom_A(P_1,P_0)$,
and
Theorem~\ref{thm:main0}(ii) holds
for each \emph{generic} $\tau$-regular module $M$.
Theorem~\ref{thm:main0} can be seen as a refinement of his result.

The following is proved by combining results in \cite{AR77} and
\cite{APT92}.

\begin{Thm} \label{thm:main6A}
Let $e$ be an idempotent in $A$, and let $B := A/AeA$.
Then for $M \in \md(B)$ the following are equivalent:
\begin{itemize}\itemsep2mm

\item[(i)]
$M$ is $\tau_A$-regular;

\item[(ii)]
$M$ is $\tau_B$-regular.

\end{itemize}
\end{Thm}

\begin{Cor}\label{cor:main6B}
Let $e$ be an idempotent in $A$, and let $B := A/AeA$.
Then
$$
\irr(B) \cap \irr^\tau(A) = \irr^\tau(B).
$$
\end{Cor}

The inclusion $\subseteq$ in Corollary~\ref{cor:main6B} follows already from
\cite[Proposition~5.2]{MP23}.
The proof is very short if one uses \cite[Proposition~4.2]{AR77}.

Let $S$ be a simple projective $A$-module.
For $\calZ \in \irr(A,\bd)$
let
$$
\varepsilon_S^+(\calZ) \in \irr(A,\bd+\dimv(S))
$$ be the \emph{generic extension} of $\calZ$ by $S$.
For
$\calZ \in \irr(A,\bd+\dimv(S))$, let
$$
\varepsilon_S^-(\calZ) \subseteq \md(A,\bd)
$$
be the \emph{generic quotient} of
$\calZ$ by $S$.
In general, $\varepsilon_S^-(\calZ)$ is just a
closed irreducible subset of $\md(A,\bd)$ and not
necessarily an irreducible component.

\begin{Thm}\label{thm:main2}
Let $S$ be a simple projective $A$-module.
For a dimension vector $\bd$ let $\bd' := \bd + \dimv(S)$.
Then the maps $\varepsilon_S^+$ and $\varepsilon_S^-$ yield
mutually inverse bijections
$$
\SelectTips{cm}{}
\xymatrix{
\irr^\tau(A,\bd) \ar@/^1ex/[rr]^{\varepsilon_S^+} &&
\irr^\tau(A,\bd')
\ar@/^1ex/[ll]^{\varepsilon_S^-}.
}
$$
\end{Thm}

As an application of Corollary~\ref{cor:main6B} and Theorem~\ref{thm:main2} we get a complete classification of
generically $\tau$-regular components for triangular algebras.

We can assume without loss of generality that
$A = KQ/I$ where $KQ$ is the path algebra of a quiver $Q$ and
$I$ is an admissible ideal in $KQ$.
Then $A$ is called \emph{triangular} if $Q$ has no oriented cycles.
One can reformulate this condition by demanding that the simple
$A$-modules $S(1),\ldots,S(n)$ can be labelled such that
$\Ext_A^1(S(i),S(j)) = 0$ for all $1 \le i \le j \le n$.

\begin{Thm}\label{thm:main3}
Let $A$ be a triangular algebra as above.
Then for each dimension vector $\bd = (d_1,\ldots,d_n)$ there is a unique generically $\tau$-regular
component $\calZ_\bd$ in $\irr(A,\bd)$, namely
$$
\calZ_\bd := (\varepsilon_{S(1)}^+)^{d_1} (\varepsilon_{S(2)}^+)^{d_2}  \cdots
 (\varepsilon_{S(n)}^+)^{d_n}(0).
$$
\end{Thm}

The $0$ appearing in Theorem~\ref{thm:main3} denotes the trivial component which contains only the zero module.

The following theorem answers the question when all
irreducible components are generically $\tau$-regular.

\begin{Thm}\label{thm:main1}
The following are equivalent:
\begin{itemize}\itemsep2mm

\item[(i)]
$\irr^\tau(A) = \irr(A)$;

\item[(ii)]
$\irr^{\tau^-}(A) = \irr(A)$;

\item[(iii)]
$A$ is hereditary.

\end{itemize}
\end{Thm}

Rougly speaking, the Auslander-Reiten translation $\tau$ yields a
bijection
$\irr^\tau(A) \to \irr^{\tau^-}(A)$.
(The precise statement generalizes
Adachi, Iyama and Reiten's \cite{AIR14}
bijection between the set of (isoclasses of) $\tau$-rigid pairs and
$\tau^-$-rigid pairs and can be found in
Section~\ref{subsec:bijectionpairs}.)

It is natural to ask when $\irr^\tau(A) = \irr^{\tau^-}(A)$.

The next theorem follows directly from Corollary~\ref{cor:main6B}
combined with some
homological formulas from Auslander-Reiten theory.

\begin{Thm}\label{thm:main4}
Suppose that $\irr^\tau(A) = \irr^{\tau^-}(A)$.
Then $A/AeA$ is a $1$-Iwanaga-Gorenstein algebra for each
idempotent $e$ in $A$.
\end{Thm}

As a consequence of
Theorems~\ref{thm:main1} and \ref{thm:main4}
we immediately get the following:

\begin{Cor}\label{cor:main5}
Assume that $\gldim(A) < \infty$.
Then the following are equivalent:
\begin{itemize}\itemsep2mm

\item[(i)]
$\irr^\tau(A) = \irr^{\tau^-}(A)$;

\item[(ii)]
$A$ is hereditary.

\end{itemize}
\end{Cor}

The following two propositions give a characterization for
$\irr^\tau(A) = \irr^{\tau^-}(A)$ for
some further classes of algebras.

\begin{Prop}\label{prop:main7}
Let $A = KQ/I$ be a gentle algebra such that $Q$ has no loops.
Then the following are equivalent:
\begin{itemize}\itemsep2mm

\item[(i)]
$\irr^\tau(A) = \irr^{\tau^-}(A)$;

\item[(ii)]
$A$ is a gentle surface algebra.

\end{itemize}
\end{Prop}

\begin{Prop}\label{prop:main8}
Let $A = KQ/I$ be a cyclic Nakayama algebra.
Then the following are equivalent:
\begin{itemize}\itemsep2mm

\item[(i)]
$\irr^\tau(A) = \irr^{\tau^-}(A)$;

\item[(ii)]
$I = J^{(n-1) + nr}$ for some $r \ge 0$ (with $r \ge 1$ if
$n=2$, and $r \ge 2$ if $n=1$).

\end{itemize}
\end{Prop}

In Proposition~\ref{prop:main8}, $J$ denotes the ideal in $KQ$ generated by all arrows of $Q$, and $n$ is the number of vertices
of $Q$.

The next theorem is a reformulation of
some results
in \cite[Section~10]{DWZ10}.
We use it for
the implications (ii) $\implies$ (i) of Propositions~\ref{prop:main7} and \ref{prop:main8}.

\begin{Thm}[Derksen, Weyman, Zelevinsky]
Let $Q$ be a $2$-acyclic quiver, and let $W$ be a potential
for $Q$ such that $W \in KQ$, and the ideal $I$ generated by the cyclic derivatives of
$W$ is an admissible ideal in $KQ$.
Let $A = J(Q,W) = KQ/I$ be the associated Jacobian algebra.
Then
$$
\hom_A(M,\tau_A(M)) = \hom_A(\tau_A^-(M),M)
$$
for all $M \in \md(A)$.
In particular,
$$
\irr^\tau(A) = \irr^{\tau^-}(A).
$$
\end{Thm}

\subsection{Conventions}
For maps $f\df U \to V$ and $g\df V \to W$ we denote their
composition by $gf\df U \to W$.

The identity map of a set $U$ is denoted by $1_U$.

For a linear map $f$ of vector spaces, let $\rk(f) := \dim \Ima(f)$
be its rank.

We assume that the set $\N$ of natural numbers includes $0$.

\subsection{Organization of the article}
In
Section~\ref{sec:preliminary} we recall
some definitions and facts on the representation theory of finite-dimensional
algebras (Auslander-Reiten formulas, $g$-vectors, minimal
projective presentations).
Section~\ref{sec:varieties} recalls (mostly) known facts on
varieties of modules over finite-dimensional algebras
(semicontinuous maps,
semicontinuity of projective presentations, generic extensions of irreducible components).

In Section~\ref{sec:tauregintro} we discuss the
concepts of reduced, regular and $\tau$-regular modules.
In particular, the section
contains the definition of generically $\tau$-regular components.

Section~\ref{sec:taureg} is dedicated to
two crucial theorems by Plamondon which characterize and parametrize generically $\tau$-regular components.
It also contains an outline of their proof (following a simplified approach by Derksen and Fei).

Theorem~\ref{thm:main0} (= Theorem~\ref{thm:main0b}) is proved in
Section~\ref{sec:maximalrank}.

Section~\ref{sec:reduction} contains the proof of Theorem~\ref{thm:main6A} (= Theorem~\ref{thm:reduction1}).

Section~\ref{sec:extensions} is dedicated to the proof of Theorem~\ref{thm:main2}
(= Theorem~\ref{thm:main2b}).

Section~\ref{sec:triangular} contains the proof of
Theorem~\ref{thm:main3} (= Theorem~\ref{thm:main3b}).

Theorem~\ref{thm:main1} (= Theorem~\ref{thm:main1b})
and its counterpart
Proposition~\ref{prop:local} are proved in
Section~\ref{sec:alltaureg}.

Finally,
Section~\ref{sec:tauversustauminus} deals with the comparison between
generically $\tau$-regular and generically $\tau^-$-regular components.
In particular, it contains the proofs of
Theorem~\ref{thm:main4} (= Theorem~\ref{thm:main4b}), Corollary~\ref{cor:main5} and Propositions~\ref{prop:main7}
and \ref{prop:main8}.


\section{Finite-dimensional algebras and homological formulas}\label{sec:preliminary}


\subsection{Finite-dimensional algebras}
Let $K$ be an algebraically closed field.
By $A$ we always denote a finite-dimensional associative $K$-algebra.
By a \emph{module} we always mean a finite-dimensional left $A$-module.
Let $\md(A)$ be the category of $A$-modules,
and let $\ind(A)$ be the subcategory of indecomposable
$A$-modules.
As usual $\tau = \tau_A$ denotes the Auslander-Reiten translation
for $\md(A)$.

Let $S(1),\ldots,S(n)$ be the simple $A$-modules, up to isomorphism.
Furthermore, let $P(1),\ldots,P(n)$ (resp. $I(1),\ldots,I(n)$)
be the indecomposable projective (resp. indecomposable injective)
$A$-modules, up to isomorphism.
These are numbered such that $\tp(P(i)) \cong S(i) \cong \soc(I(i))$
for $1 \le i \le n$.

For $M \in \md(A)$ the
\emph{Jordan-H\"older multiplicity} of $S(i)$ in $M$ is denoted by
$[M:S(i)]$.
We have $\dim \Hom_A(P(i),M) = [M:S(i)]$.
Then
$$
\dimv(M) := ([M:S(1)],\ldots,[M:S(n)])
$$
is the
\emph{dimension vector} of $M$.
Let
$$
\supp(M) := \supp(\dimv(M)) :=
\{ 1 \le i \le n \mid [M:S(i)] \not= 0 \}
$$
be the
\emph{support} of $M$.

Let $\proj(A)$ (resp. $\inj(A)$) be the subcategory of projective (resp. injective) $A$-modules.

Let $D := \Hom_K(-,K)$ be the usual duality.

The \emph{Nakayama functor}
$$
\nu_A := D\Hom_A(-,{_A}A)\df \proj(A) \to \inj(A)
$$
is an equivalence of subcategories.

For each $P \in \proj(A)$ we have
$$
P \cong P(1)^{a_1} \oplus \cdots \oplus P(n)^{a_n}
$$
for some $a_1,\ldots,a_n \ge 0$.
Let $[P] := (a_1,\ldots,a_n)$, and set $[P:P(i)] := a_i$ for $1 \le i \le n$.

For some backround on Auslander-Reiten theory we refer to \cite{ARS97}.

For $M,N \in \md(A)$ and $i \ge 1$ let
$$
\hom_A(M,N) := \dim \Hom_A(M,N)
\text{\quad and \quad}
\ext_A^i(M,N) := \dim \Ext_A^i(M,N).
$$
Furthermore, let
$$
r(M,N) := \max\{ \rk(f) \mid f \in \Hom_A(M,N) \}
$$
and
$$
\Hom_A(M,N)^\circ :=
\{ f \in \Hom_A(M,N) \mid \rk(f) = r(M,N) \}.
$$

\subsection{Some formulas for the Auslander-Reiten translation}
For a proof (and the missing definitions) of the following theorem
we refer to \cite[Section~IV.4]{ARS97} or
\cite[Section~III.6]{SY11}.

\begin{Thm}[{Auslander, Reiten}]\label{thm:ARformulas}
For $M,N \in \md(A)$ there are functorial isomorphisms
$$
D\overline{\Hom}_A(N,\tau_A(M)) \cong \Ext_A^1(M,N) \cong D\underline{\Hom}_A(\tau_A^-(N),M).
$$
\end{Thm}

\begin{Lem}[{\cite[Section~IV.4]{ARS97}}]\label{lem:pdim1stable}
Let $M,N \in \md(A)$.
If $\pdim(N) \le 1$, then
$$
\overline{\Hom}_A(M,\tau_A(N)) = \Hom_A(M,\tau_A(N)).
$$
\end{Lem}

\begin{Lem}[{\cite[Section~III.5]{SY11}}]\label{lem:pdim1}
For $M \in \md(A)$
the following are equivalent:
\begin{itemize}\itemsep2mm

\item[(i)]
$\pdim(M) \le 1$;

\item[(ii)]
For each $I \in \inj(A)$ we have
$\Hom_A(I,\tau_A(M)) = 0$.

\end{itemize}
\end{Lem}

There are also duals of Lemmas~\ref{lem:pdim1stable}
and \ref{lem:pdim1}.

\subsection{$\tau$-rigid modules}
A module $M \in \md(A)$ is
\emph{$\tau$-rigid} (resp. \emph{$\tau^-$-rigid})
if
$$
\Hom_A(M,\tau_A(M)) = 0
\qquad (\text{resp. } \Hom_A(\tau_A^-(M),M) = 0).
$$

Let $\tau$-$\rigid(A)$ (resp. $\tau^-$-$\rigid(A)$) be the subcategory of all $\tau$-rigid (resp. $\tau^-$-rigid) $A$-modules.

\subsection{$g$-vectors}\label{subsec:gvectors}
For $M \in \md(A)$ let
$$
P_1^M \xrightarrow{f_M} P_0^M \to M \to 0
$$
always denote
a minimal projective presentation of $M$.
(This is uniquely determined, up to isomorphism of complexes.)
Then
$$
g(M) := (g_1(M),\ldots,g_n(M)) := [P_1^M] - [P_0^M]
$$
is the \emph{$g$-vector} of $M$.

\begin{Lem}For $M \in \md(A)$ and
$1 \le i \le n$ we have
$$
g_i(M) = -\hom_A(M,S(i)) + \ext_A^1(M,S(i)).
$$
\end{Lem}

\begin{proof}
One easily checks that
$\hom_A(M,S(i)) = [P_0^M:P(i)]$ and
$\ext_A^1(M,S(i)) = [P_1^M:P(i)]$.
The claim follows.
\end{proof}

Let $(P_1,P_0)$ be a pair of projective $A$-modules.
For $f \in \Hom_A(P_1,P_0)$ let
$P_f$ be a maximal direct summand of $P_1$ such that
$f(P_f) = 0$.

The next lemma is folklore.

\begin{Lem}\label{lem:gvector1}
For $M \in \md(A)$
let
$$
P_1 \xrightarrow{f} P_0 \to M \to 0
$$
be a projective presentation of $M$.
Then $P_1 \xrightarrow{f} P_0$ is isomorphic
(in the category of complexes) to a
direct sum
$$
(P_1^M \xrightarrow{f_M} P_0^M) \oplus
(P \xrightarrow{1_P} P) \oplus
(P_f \to 0).
$$
Up to isomorphism, these three summands are uniquely determined
by $f$.
The projective presentation is minimal if and only if
$P = 0$ and $P_f = 0$.
We have
$$
g(M) =  [P_1^M] - [P_0^M] = [P_1] - [P_0] - [P_f].
$$
\end{Lem}

\begin{Lem}
\label{lem:gvector3}
Let $M,N \in \md(A)$, and let
$$
P_1 \xrightarrow{f} P_0 \to M \to 0
$$
be a projective presentation.
Applying $\Hom_A(-,N)$ yields an exact sequence
\begin{multline*}
0 \to \Hom_A(M,N) \to \Hom_A(P_0,N) \to \Hom_A(P_1,N)
\\
\to \Hom_A(N,\tau_A(M)) \oplus \Hom_A(P_f,M) \to 0.
\end{multline*}
\end{Lem}

\begin{proof}
This follows from \cite[Section~IV.4]{ARS97}
combined with
Lemma~\ref{lem:gvector1}.
\end{proof}

Let $K_0(\proj(A))$ and $K_0(A)$ be the Grothendieck groups of
$\proj(A)$ and $\md(A)$, respectively.
(Both are isomorphic to $\Z^n$.
As a $\Z$-basis they have $\{ [P(i)] \mid 1 \le i \le n \}$ and
$\{ [S(i)] \mid 1 \le i \le n \}$, respectively.
For $M \in \md(A)$, we can identify $[M] \in K_0(A)$ and $\dimv(M)$.)

Then
\begin{align*}
\bil{-,?}\df K_0(\proj(A)) \times K_0(A) &\to \Z
\\
([P],[M]) &\mapsto \hom_A(P,M)
\end{align*}
defines a $\Z$-bilinear form.

Also the following lemma is well known (and follows from Lemmas~\ref{lem:gvector1} and~\ref{lem:gvector3}).

\begin{Lem}\label{lem:bilinear}
For $M,N \in \md(A)$ we have
$$
\bil{g(M),[N]} = - \hom_A(M,N) + \hom_A(N,\tau_A(M)).
$$
\end{Lem}


\section{Varieties of modules and semicontinuity of projective presentations}\label{sec:varieties}


\subsection{Varieties of modules}
From now on we assume
that
$$
A = KQ/I
$$
where $Q = (Q_0,Q_1,s,t)$
is a finite quiver with vertex set $Q_0 = \{ 1,\ldots,n \}$
and arrow set $Q_1$,
and $I$ is an admissible ideal in the path algebra
$KQ$.
(In contradiction to our global assumption, the algebra $KQ$ might be infinite-dimensional. But $KQ/I$ is finite-dimensional.)
Let $e_1,\ldots,e_n$ be the paths of length $0$ in $Q$.
They form a complete set of pairwise orthogonal primitive
idempotents in $A$.
An arrow $a \in Q_1$ starts in $s(a)$ and ends in $t(a)$.

Each finite-dimensional algebra over an algebraically closed field $K$
is Morita equivalent to such an algebra $KQ/I$.

For $\bd = (d_1,\ldots,d_n) \in \N^n$
let
$$
\rep(Q,\bd) :=
\prod_{a \in Q_1} \Hom_K(K^{d_{s(a)}},K^{d_{t(a)}})
$$
be the affine space of \emph{representations} of $Q$ with
\emph{dimension vector} $\bd$.
Define
$$
G_\bd := \prod_{i=1}^n \GL_{d_i}(K).
$$
For $g = (g_1,\ldots,g_n) \in G_\bd$ and
$M = (M_a)_{a \in Q_1} \in \rep(Q,\bd)$ let
$$
g.M := (g_{t(a)}M_ag_{s(a)}^{-1})_{a \in Q_1}.
$$
This defines a $G_\bd$-action on $\rep(Q,\bd)$, and the
isomorphism classes of representations with dimension vector
$\bd$ correspond to the orbits of this action.
The orbit of $M \in \rep(Q,\bd)$ is denoted by $\cO_M$.

For $A = KQ/I$ let
$$
\md(A,\bd) := \{ M \in \rep(Q,\bd) \mid IM = 0 \}.
$$
This is a  closed stable subset of $\rep(Q,\bd)$.
(For a set $X$ with a group $G$ acting on $X$,
a subset $U$ of $X$ is called \emph{stable} if $U$ consists of
a union of orbits of the $G$-action.)
We call $\md(A,\bd)$ the \emph{variety of $A$-modules} with dimension
vector $\bd$.

Let
$\Spec(A,\bd)$ be the affine scheme
defined by the equations given by the condition $IM = 0$ for all $M \in \rep(Q,\bd)$.
Then $\md(A,\bd)$ can be seen as the associated reduced scheme.

We rarely use the scheme structure of $\Spec(A,\bd)$
and prefer to work with the affine variety $\md(A,\bd)$.
The only (but important) exception is the tangent space
$T_M$ of a module $M$ at the scheme $\Spec(A,\bd)$.
We call $M$ \emph{regular} if it is regular
as a closed point in $\Spec(A,\bd)$.

One can also define varieties of modules without using quivers
and admissible ideals.
Namely, for our finite-dimensional algebra $A$, and $d \in \N$
let $\md(A,d)$ be the affine variety of $K$-algebra homomorphims
$A \to M_n(K)$.
The precise relation between the varieties $\md(A,d)$ and
$\md(A,\bd)$ is discussed in \cite{B91}.
The definition of $\md(A,d)$ is more straightforward and does not need the passage from $A$ to $KQ/I$, but $\md(A,\bd)$ makes it
easier to work with explicit examples.
For some fundamental results on $\md(A,d)$ we refer to
\cite{G74}.

\subsection{Irreducible components}\label{subsec:irreducible}
For $\bd = (d_1,\ldots,d_n) \in \N^n$
let $\irr(A,\bd)$ be the set of irreducible components of
the affine variety $\md(A,\bd)$.
Let
$$
\irr(A) := \bigcup_{\bd \in \N^n} \irr(A,\bd).
$$
For
$\calZ \in \irr(A,\bd)$ let
$$
\supp(\calZ) := \supp(\bd) := \{ 1 \le i \le n \mid d_i \not= 0 \}
$$
be the \emph{support} of $\calZ$.
Let $\dimv(\calZ) := \bd$, and let $[\calZ:S(i)] := d_i$ for
$1 \le i \le n$.
Let
$$
\calZ^{\rm int} := \{ M \in \calZ \mid M \text{ is not contained in any other
component in $\irr(A,\bd)$} \}
$$
be the \emph{interior} of $\calZ$.

For $M \in \md(A,\bd)$ recall that
$$
\md_M(A,\bd) :=
\bigcup_{\substack{\calZ \in \irr(A,\bd)\\M \in \calZ}} \calZ.
$$

\noindent{\bf Examples}:
\begin{itemize}\itemsep2mm

\item[(i)]
Let $M \in \md(A,\bd)$ with $\Ext_A^1(M,M) = 0$.
Then $\calZ := \overline{\cO_M} \in \irr(A,\bd)$.

\item[(ii)]
Let $A = KQ/I$ where $Q$ is the quiver
$$
\xymatrix{
1 & 2 \ar[l]_a & 3 \ar[l]_b
}
$$
and $I$ is generated by $ab$.
Let $\bd = (d_1,d_2,d_3) \in \N^3$.
Then
$$
\md(A,\bd)
= \{ (M_a,M_b) \in \rep(Q,\bd) \mid M_aM_b = 0 \}.
$$
The irreducible components
of $\md(A,\bd)$ are parametrized by
the maximal
pairs $(r_a,r_b) \in \N^2$ such that
$r_a \le d_1$,  $r_a+r_b \le d_2$ and $r_b \le d_3$.
For example for $\bd = (4,6,5)$ the maximal pairs are
$(4,2)$, $(3,3)$, $(2,4)$ and $(1,5)$.
For a maximal pair $(r_a,r_b)$ the corresponding irreducible component of $\md(A,\bd)$
is
$$
\calZ_{(r_a,r_b)} := \{ (M_a,M_b) \in \md(A,\bd) \mid
\rk(M_a) \le r_a,\; \rk(M_b) \le r_b \}.
$$
\end{itemize}

\subsection{Semicontinuous maps on varieties of modules}
A map $\eta\df X \to \Z$ on an affine variety $X$ is
\emph{upper semicontinuous} (resp. \emph{lower semicontinuous})
if
$$
X_{\le i} := \{ x \in X \mid \eta(x) \le i \} \quad
(\text{resp. }
X_{\ge i} := \{ x \in X \mid \eta(x) \ge i \})
$$
is an open subset of $X$ for each $i \in \Z$.
If $\eta\df X \times X \to \Z$ is upper or lower semicontinuous,
then the associated diagonal map $x \mapsto \eta(x,x)$ is
also upper or lower semicontinuous, respectively.
Similarly, if $\eta\df X \times Y \to \Z$ is upper or lower semicontinuous,
then for each $x' \in X$ and $y' \in Y$ the maps
$x \mapsto \eta(x,y')$ and $y \mapsto \eta(x',y)$ are
also upper or lower semicontinuous, respectively.

If $X$ is irreducible and $\eta\df X \to \Z$ is an upper
semicontinuous map whose image has a minimum, then
$\{ x \in X \mid \eta(x) \text{ is minimal} \}$
is a dense open subset of $X$.
We call its elements \emph{generic} (with respect to
$\eta$) in $X$, and the minimal value is the
\emph{generic value} of $\eta$ on $X$.
Analogously, for lower semicontinuous maps
one needs to consider the maximal value.

In our context we study semicontinuous maps for $X = \md(A,\bd)$
or $X = \md(A,\bd) \times \md(A,\bd')$,
and we often restrict them to some
$\calZ$ or $\calZ \times \calZ'$ with $\calZ,\calZ' \in \irr(A)$.

The maps
\begin{align*}
(M,N) &\mapsto \hom_A(M,N),
&
(M,N) & \mapsto \ext_A^i(M,N),
\\
(M,N) & \mapsto \hom_A(M,\tau_A(N)),
&
(M,N) & \mapsto \hom_A(\tau_A^-(M),N),
\end{align*}
are all upper semicontinuous on
$\md(A,\bd) \times \md(A,\bd')$,
and the maps
\begin{align*}
M &\mapsto \pdim(M),
&
M & \mapsto g_i(M) := - \hom_A(M,S(i)) + \ext_A^1(M,S(i)),
\\
M & \mapsto \idim(M)
\end{align*}
are upper semicontinuous on
$\md(A,\bd)$.
The map
$f \mapsto \rk(f)$ is
lower semicontinuous on $\Hom_A(M,N)$ for all
$M,N \in \md(A)$.
The map $M \mapsto \dim \cO_M$ is lower semicontinuous
on $\md(A,\bd)$.
Note that
$$
\dim \cO_M = \dim G_\bd - \hom_A(M,M).
$$
We refer to \cite{GLFS23} for more details on some of the maps above.

For $\calZ \in \irr(A)$ let
$$
\pdim(\calZ) := \min\{ \pdim(M) \mid M \in \calZ \}.
$$

For $\calZ \in \irr(A)$ and $1 \le i \le n$ let
$$
g_i(\calZ) := \min\{ g_i(M) \mid M \in \calZ \}
$$
be the generic value of $g_i(-)$ on $\calZ$.
Then
$$
g(\calZ) := (g_1(\calZ),\ldots,g_n(\calZ))
$$
is the \emph{$g$-vector} of $\calZ$.

\subsection{Dimension of fibres}
We need the following standard result from basic algebraic geometry.

\begin{Lem}\label{lem:chevalley}
Let $f\df X \to Y$ be a dominant morphism of
irreducible quasi-projective varieties.
Then for each $y \in Y$ we have
$$
\dim f^{-1}(y) \ge \dim(X) - \dim(Y).
$$
Furthermore, there is a dense open subset $\calU$ of $Y$ such that
$$
\dim f^{-1}(u) = \dim(X) - \dim(Y)
$$
for all $u \in \calU$.
\end{Lem}

\subsection{Upper semicontinuity of projective presentations}
The following lemma and its proof are due to
Fei \cite{F23b}.
We use it frequently.

\begin{Lem}[{Fei}]\label{lem:presentationsemicont}
Let $\calZ$ be a closed irreducible subset of $\md(A,\bd)$.
Let $M \in \calZ$, and let
$$
P_1 \to P_0 \to M \to 0
$$ be
a projective presentation of $M$.
Then there is a dense open subset $\calU$ of $\calZ$ such that
$M \in \calU$ and
each $M' \in \calU$ has a projective presentation of the form
$$
P_1 \to P_0 \to M' \to 0.
$$
\end{Lem}

\begin{proof}
Let us first
assume that the presentation $P_1 \to P_0 \to M \to 0$
is minimal.
The maps $\hom_A(-,S(j))$ and $\ext_A^1(-,S(j))$ are upper
semicontinuous, and the map
$\hom_A(-,S(j)) - \ext_A^1(-,S(j))$ is lower semicontinuous.
Let $\calU$ be the set of all $N \in \calZ$ such that
\begin{align*}
a_j' := \ext_A^1(N,S(j)) &\le \ext_A^1(M,S(j)) =: a_j,
\\
b_j' := \hom_A(N,S(j)) &\le \hom_A(M,S(j)) =: b_j,
\\
\hom_A(N,S(j)) -\ext_A^1(N,S(j)) &\ge \hom_A(M,S(j)) - \ext_A^1(M,S(j))
\end{align*}
for all $1 \le j \le n$.
Thus $\calU$ is the intersection of three non-empty open subsets of
the closed irreducible subset $\calZ$.
Therefore $\calU$ is open and dense in $\calZ$.

For $N \in \calU$ let
$P_1' \xrightarrow{g_1} P_0' \to N \to 0$ be a minimal projective presentation.
Thus we have
\begin{align*}
P_1 &= \bigoplus_{j=1}^n P(j)^{a_j},
&
P_0 &= \bigoplus_{j=1}^n P(j)^{b_j},
\\
P_1' &= \bigoplus_{j=1}^n P(j)^{a_j'},
&
P_0' &= \bigoplus_{j=1}^n P(j)^{b_j'}.
\end{align*}
For $1 \le j \le n$ we have $a_j' \le a_j$, $b_j' \le b_j$ and
$b_j' - a_j' \ge b_j - a_j$.
The last inequality is equivalent to $a_j-a_j' \ge b_j-b_j'$.
Thus for
$$
P_1'' := \bigoplus_{j=1}^n P(j)^{a_j-a_j'}
\text{\quad and \quad}
P_0'' := \bigoplus_{j=1}^n P(j)^{b_j-b_j'}
$$
there is an epimorphism $g_2\df P_1'' \to P_0''$.
Observe that $P_1 = P_1' \oplus P_1''$ and $P_0 = P_0' \oplus P_0''$.
We get a projective presentation
$$
\SelectTips{cm}{}
\xymatrix@+1pc{
P_1' \oplus P_1'' \ar[r]^{\left(\bsm g_1&0\\0&g_2\esm\right)} &
P_0' \oplus P_0'' \ar[r] & N \ar[r] & 0.
}
$$

Finally,
each non-minimal projective presentation $P_1 \to P_0 \to M \to 0$
has the minimal presentation of $M$ as a direct summand, see
Lemma~\ref{lem:gvector1}.
Now one argues as before and gets the result also in this case.
\end{proof}

The next lemma is well known.
It
follows from the upper semicontinuity of
the maps $\hom_A(-,S(i))$ and $\ext_A^1(-,S(i))$.

\begin{Lem}\label{lem:genericpresentations1}
For $\calZ \in \irr(A)$ there is a unique (up to isomorphism) pair
$(P_1^{\calZ},P_0^{\calZ})$
of projective $A$-modules such that for a dense open subset
$\calU$ of $\calZ$ each $M \in \calU$ has a minimal projective presentation of the form
$$
P_1^{\calZ} \to P_0^{\calZ} \to M \to 0.
$$
\end{Lem}

Note that it can happen that for different components $\calZ_1$
and $\calZ_2$ in $\irr(A)$ we have
$(P_1^{\calZ_1},P_0^{\calZ_1}) = (P_1^{\calZ_2},P_0^{\calZ_2})$.

Here is the corresponding dual statement:

\begin{Lem}\label{lem:genericpresentations2}
For $\calZ \in \irr(A)$ there is a unique (up to isomorphism) pair
$(I_0^{\calZ},I_1^{\calZ})$
of injective $A$-modules such that for a dense open subset
$\calU$ of $\calZ$ each $M \in \calU$ has a minimal injective presentation of the form
$$
0 \to M \to I_0^{\calZ} \to I_1^{\calZ}.
$$
\end{Lem}

\subsection{Maximal rank and dimension vectors}
Also the following lemma is well known.
For convenience we include a proof.

\begin{Lem}\label{lem:maximalrank}
For $M,N \in \md(A)$ let $f,g \in \Hom_A(M,N)$ with
$$
\rk(f) = \rk(g) = r(M,N).
$$
Then $\dimv(\Ima(f)) = \dimv(\Ima(g))$.
\end{Lem}

\begin{proof}
Let $f \in \Hom_A(M,N)$.
For $1 \le i \le n$ we have $f(e_iM) \subseteq e_iN$.
Let $f_i\df e_iM \to e_iN$ denote the restriction map.
Thus we can write
$$
f = \left(\bbm f_1 &&\\&\ddots&\\&&f_n\ebm\right)\df \bigoplus_{i=1}^n e_iM
\to \bigoplus_{i=1}^n e_iN.
$$
We have $\rk(f_i) = [\Ima(f):S(i)]$ and
$\rk(f) = \rk(f_1) + \cdots +\rk(f_n)$.
In particular, $\dimv(\Ima(f)) = (\rk(f_1),\ldots,\rk(f_n))$.

Let
$$
r_i := \max\{ \rk(f_i) \mid f \in \Hom_A(M,N) \}.
$$
Thus
$\rk(f) \le r_1 + \cdots + r_n$ for all $f \in \Hom_A(M,N)$.
By lower semicontinuity,
$$
\calU_i := \{ f \in \Hom_A(M,N) \mid \rk(f_i) = r_i \}
$$
is a dense open subset of $\Hom_A(M,N)$.
The intersection $\calU_1 \cap \cdots \cap \calU_n$ is
again dense open in $\Hom_A(M,N)$.
It follows that $r(M,N) = r_1 + \cdots + r_n$.
This yields the claim of the lemma.
\end{proof}

\subsection{Projective presentations of maximal rank}

\begin{Lem}\label{lem:maxrankfM}
For $M \in \md(A)$ let
$$
P_1 \xrightarrow{f} P_0 \to M \to 0
$$
be a projective presentation.
Then the following are equivalent:
\begin{itemize}\itemsep2mm

\item[(i)]
$\rk(f) = r(P_1,P_0)$;

\item[(ii)]
$\rk(f_M) = r(P_1^M,P_0^M)$ and
$\Hom_A(P_f,M) = 0$.

\end{itemize}
\end{Lem}

\begin{proof}
Using Lemma~\ref{lem:gvector1} we can assume that
$$
(P_1 \xrightarrow{f} P_0) =
(P_1^M \xrightarrow{f_M} P_0^M) \oplus
(P \xrightarrow{1_P} P) \oplus
(P_f \to 0).
$$

(i) $\implies$ (ii):
Let $\rk(f) = r(P_1,P_0)$.
Assume there is some $g' \in \Hom_A (P_1^M, P_0^M)$
with $\rk(g') > \rk(f_M)$.
Let
$$
g :=
\left(
\bbm
g' & 0 & 0 \\ 0 & 1_P & 0
\ebm\right)\df P_1 = P_1^M \oplus P \oplus P_f \to
P_0^M \oplus P = P_0,
$$
i.e.\ $g$ is the direct sum
$$
(P_1^M \xrightarrow{g'} P_0^M) \oplus (P \xrightarrow{1_P} P) \oplus (P_f \to 0).
$$
Then $\rk(g) > \rk(f)$, a contradiction.

Next, let $h' \in \Hom_A (P_f,M)$ be non-zero.
Let $\pi \colon P_0^M \to M$ be the epimorphism with kernel
$\Ima(f_M)$.
By the projectivity of $P_f$ there exists some
$h \in \Hom_A(P_f,P_0^M)$ such that $h' = \pi h$.
Define
$$
g :=
\left(
\bbm
f_M & 0 & h \\ 0 & 1_P & 0
\ebm\right)\df P_1^M \oplus P \oplus P_f \to
P_0^M \oplus P.
$$
Then one easily checks that $\Ima(f)$ is a proper submodule of $\Ima(g)$.
Therefore we have $\rk(f)<\rk(g)$, again a contradiction.

(ii) $\implies$ (i):
Assume that $\rk(f_M) = r(P_1^M,P_0^M)$ and
$\Hom_A(P_f,M) = 0$.
Since $P_f$ is a projective module, the latter condition means that
$\Ima(g') \subseteq \Ima(f')$ for all $g'\in \Hom_A(P_f,P_0^M)$ and
$f' \in \Hom_A(P_1^M,P_0^M)^\circ$.
(Here we use that $\dimv (\Coker (f')) = \dimv (M)$ by Lemma~\ref{lem:maximalrank}.)

It follows from \cite[Corollary~4.2]{DF15} and the lower semicontinuity of the rank that there exists
$g \in \Hom_A(P_1,P_0)^\circ$ such that
$$
(P_1 \xrightarrow{g} P_0) =
(P_1^M \oplus P_f \xrightarrow{(f',g')} P_0^M)
\oplus (P \xrightarrow{1_P} P)
$$
with $f' \in \Hom_A (P_1^M,P_0^M)^\circ$.
Then $\Ima(g') \subseteq \Ima(f')$, and therefore
$\rk(f',g') = \rk (f') = \rk (f_M)$.
Thus
$$
r(P_1,P_0) = \rk(g) = \rk(f',g') + \dim(P)
= \rk(f_M) + \dim(P) = \rk(f).
$$
\end{proof}

\subsection{Generic extensions of components}
\label{subsec:genericextension}
For $i=1,2$ let $\calZ_i \in \irr(A,\bd_i)$, and let
$\bd := \bd_1 + \bd_2$.
Let
$$
\ext_A^1(\calZ_1,\calZ_2) := \min\{ \ext_A^1(M_1,M_2) \mid
(M_1,M_2) \in \calZ_1 \times \calZ_2 \}.
$$
Let
$$
\calU := \{ (M_1,M_2) \in \calZ_1 \times \calZ_2 \mid \ext_A^1(M_i,M_j) = \ext_A^1(\calZ_i,\calZ_j)
\text{ for } \{ i,j \} = \{ 1,2 \} \}.
$$
By upper semicontinuity, $\calU$ is a dense open subset of
$\calZ_1 \times \calZ_2$.
Let
$\cE(\calU)$ be the set of all $E \in \md(A,\bd)$ such that there exists
a short exact sequence
$$
0 \to M_2 \to E \to M_1 \to 0
$$
where $(M_1,M_2) \in \calU$.
Let $$
\cE(\calZ_1,\calZ_2) := \overline{\cE(\calU)}
$$ be the
\emph{generic extension} of $\calZ_1$ by $\calZ_2$.
This is a closed irreducible subset of $\md(A,\bd)$.
However, in general $\cE(\calZ_1,\calZ_2)$ is not an irreducible
component.

\begin{Thm}[{\cite[Corollary~1.4]{CBS02}}]\label{thm:CBS}
Let $\calZ_1,\calZ_2 \in \irr(A)$.
If $\ext_A^1(\calZ_2,\calZ_1) = 0$, then
$\cE(\calZ_1,\calZ_2) \in \irr(A)$.
\end{Thm}

For example, if $\calZ_2 = \overline{\cO_P}$ with $P \in \proj(A)$,
then
the assumption of Theorem~\ref{thm:CBS} is satisfied and we get
$\cE(\calZ_1,\calZ_2) \in \irr(A)$.


\section{Generically $\tau$-regular
components}\label{sec:tauregintro}


\subsection{Generically reduced components}
For $M \in \md(A,\bd)$ let $T_M^{\rm red}$ be the tangent space of $M$ at the affine variety $\md(A,\bd)$, and let
$T_M$ be the tangent space of $M$ at the affine scheme
$\Spec(A,\bd)$.
Recall that one treats $\md(A,\bd)$ as the reduced scheme associated with $\Spec(A,\bd)$.
The module $M$ is \emph{regular} as a closed point of
$\Spec(A,\bd)$ if
$$
\dim T_M = \dim \md_M(A,\bd).
$$
(In general, we have $\dim T_M \ge \dim \md_M(A,\bd)$.)

Let $T_M(\cO_M)$ be the tangent space of $M$ at the orbit
$\cO_M$.
We have
$$
\dim T_M(\cO_M) = \dim \cO_M.
$$

The reduced tangent space $T_M^{\rm red}$ is usually quite inaccessible.
But $T_M$ can often be computed by the following classical result.

\begin{Prop}[{{Voigt's Lemma} \cite[Section~1.1]{G74}}]
\label{prop:voigt}
There is an isomorphism of vector spaces
$$
T_M/T_M(\cO_M) \to \Ext_A^1(M,M).
$$
\end{Prop}

We call $M \in \md(A,\bd)$ \emph{reduced} if $\dim T_M = \dm T_M^{\rm red}$.

An irreducible components $\calZ \in \irr(A)$ is \emph{generically
reduced} if there exists a dense open subset of $\calZ$ which
consists only of reduced modules.

\subsection{Definition of generically $\tau$-regular components}
\label{subsec:deftaureg}

For $M \in \md(A,\bd)$ define
\begin{align*}
c(M) &:=  c_A(M) := \dim \md_M(A,\bd) - \dim \cO_M,
\\
e(M) &:= e_A(M) := \ext_A^1(M,M),
\\
E(M) &:= E_A(M) := \hom_A(M,\tau_A(M)).
\end{align*}

For $\calZ \in \irr(A)$ let
\begin{align*}
c(\calZ) &:= c_A(\calZ) := \min\{ c(M) \mid M \in \calZ \},
\\
e(\calZ) &:= e_A(\calZ) := \min\{ e(M) \mid M \in \calZ \},
\\
E(\calZ) &:= E_A(\calZ) := \min\{ E(M) \mid M \in \calZ \}.
\end{align*}
In the same fashion, one defines $c(\calZ)$, $e(\calZ)$ and
$E(\calZ)$ for an arbitrary closed irreducible subset of
$\md(A,\bd)$.
Using Voigt's Lemma \ref{prop:voigt} and
Theorem~\ref{thm:ARformulas} one gets
$$
c(M) \le e(M) \le E(M)
\text{\quad and \quad}
c(\calZ) \le e(\calZ) \le E(\calZ).
$$

\begin{Lem}\label{lem:regular}
For
$M \in \md(A,\bd)$ the following are equivalent:
\begin{itemize}\itemsep2mm

\item[(i)]
$c(M) = e(M)$;

\item[(ii)]
$M$ is a regular point of the affine scheme
$\Spec(A,\bd)$.

\end{itemize}
In this case,
$M$ is contained in exactly one irreducible component
of $\md(A,\bd)$.
\end{Lem}

\begin{proof}
This follows from Voigt's Lemma \ref{prop:voigt} and from the fact that
points in the intersection of different irreducible components
are singular (= non-regular).
\end{proof}

A module $M$ is called \emph{regular} if $c(M) = e(M)$.
A component $\calZ \in \irr(A)$ is \emph{generically regular} if
there exists a dense open subset of $\calZ$ which consists
only of regular modules.
This is the case if and only if
$c(\calZ) = e(\calZ)$.

The following lemma is folklore.

\begin{Lem}
For $\calZ \in \irr(A)$ the following are equivalent:
\begin{itemize}\itemsep2mm

\item[(i)]
$\calZ$ is generically regular;

\item[(ii)]
$\calZ$ is generically reduced.

\end{itemize}
\end{Lem}

\begin{proof}
(i) $\implies$ (ii):
For $M \in \md(A,\bd)$
we have
$$
c(M) \le \dim T_M^{\rm red} -\dim \cO_M \le \dim T_M -
\dim \cO_M = e(M).
$$
(The equality follows from Voigt's Lemma \ref{prop:voigt}.)
Thus $c(M) = e(M)$ implies $\dim T_M^{\rm red} = \dim T_M$.

(ii) $\implies$ (i):
There is a dense open subset $\calU$ of
$\calZ$ such that $\dim T_M^{\rm red} = \dim(\calZ)$
for all $M \in \calU$.
If $\calZ$ is generically reduced, then
we can choose $\calU$ such that
$\dim T_M^{\rm red} = \dim T_M$ for all $M \in \calU$.
Now Voigt's Lemma \ref{prop:voigt} says that
$c(\calZ) = e(\calZ)$.
\end{proof}

A module  $M$ is
called
\emph{$\tau$-regular} if $c(M) = E(M)$.

Thus $\tau$-regular modules are regular.
The converse is in general wrong.

Part (i) of the following lemma is well known.

\begin{Lem}\label{lem:regularopen}
For each dimension vector $\bd$ the following hold:
\begin{itemize}\itemsep2mm

\item[(i)]
The regular modules in $\md(A,\bd)$ form an open subset of $\md(A,\bd)$;

\item[(ii)]
The $\tau$-regular modules in $\md(A,\bd)$
form an open subset of $\md(A,\bd)$.

\end{itemize}
\end{Lem}

\begin{proof}
(i):
By \cite[Theorem~1.1]{GLFS23}
the map
\begin{align*}
\eta\df \md(A,\bd) &\to \Z
\\
M &\mapsto - \hom_A(M,M) + \ext_A^1(M,M)
\end{align*}
is upper semicontinuous.
For each $M \in \md(A,\bd)$ we have
$$
\dim \md_M(A,\bd) - \dim G_\bd \le
- \hom_A(M,M) + \ext_A^1(M,M)
$$
with an equality if and only if $c(M) = e(M)$.
The upper semicontinuity of $\eta$ implies now (i).

(ii):
Again
by \cite[Theorem~1.1]{GLFS23}, for each $1 \le i \le n$
the map
\begin{align*}
\eta_i\df \md(A,\bd) &\to \Z
\\
M &\mapsto g_i(M) = -\hom_A(M,S(i)) + \ext_A^1(M,S(i))
\end{align*}
is upper semicontinuous.
For each $M \in \md(A,\bd)$ we have
$$
\dim \md_M(A,\bd) - \dim G_\bd \le
- \hom_A(M,M) + \hom_A(M,\tau_A(M))
$$
with an equality if and only if $c(M) = E(M)$.
By Lemma~\ref{lem:bilinear} we have
$$
\bil{g(M),[M]} = - \hom_A(M,M) + \hom_A(M,\tau_A(M)).
$$
Since
$$
\bil{g(M),[M]} = \hom_A(P_1^M,M) - \hom_A(P_0^M,M)
= \sum_{i=1}^n g_i(M) [M:S(i)],
$$
the upper semicontinuity of $\eta_i$ implies now (ii).
\end{proof}

We call $\calZ \in \irr(A)$ \emph{generically $\tau$-regular} if
there exists a dense open subset of $\calZ$ which consists
only of $\tau$-regular modules.
This is the case if and only if
$c(\calZ) = E(\calZ)$.

Generically $\tau$-regular components were first defined and studied in \cite{GLS12},
where they run under the name \emph{strongly reduced
components}.
They play an important role in the construction of
generic bases of Fomin-Zelevinsky cluster algebras.
Subsequently we also used to call them \emph{generically $\tau$-reduced components}.
We believe now that \emph{generically $\tau$-regular} is
the better terminology.

For the dual definition of $\tau^-$-regular modules and
generically $\tau^-$-regular components one uses
$E^-(M) := \hom_A(\tau_A^-(M),M)$.

Let $\irr^\tau(A,\bd)$ be the set
of generically $\tau$-regular components
in $\irr(A,\bd)$, and let
$$
\irr^\tau(A) := \bigcup_{\bd \in \N^n} \irr^\tau(A,\bd).
$$
Analogously, one defines $\irr^{\tau^-}(A,\bd)$ and
$\irr^{\tau^-}(A)$.

\subsection{Examples}

\subsubsection{}
Let $\calZ \in \irr(A)$ with $\pdim(\calZ) \le 1$.
Then $\calZ$ is generically $\tau$-regular.
(For $M \in \md(A)$ with $\pdim(M) \le 1$ we have
$e(M) = E(M)$, see Theorem~\ref{thm:ARformulas} and Lemma~\ref{lem:pdim1stable}.
Now the claim follows from \cite[Proposition~23]{R01}.)

\subsubsection{}
Let $M \in \md(A)$ be $\tau$-rigid,
i.e.\ $\Hom_A(M,\tau_A(M)) = 0$.
Then $\calZ := \overline{\cO_M}$ is generically $\tau$-regular.

\subsubsection{}
Let $A = KQ/I$ where $Q$ is the quiver
$$
\xymatrix{
1 & 2 \ar[l]_a & 3 \ar[l]_b
}
$$
and $I$ is generated by $ab$.
Let $\bd = (1,1,1)$.
Let
\begin{align*}
M_1 &:= S(1) \oplus P(3) = \bbm 1 \ebm \oplus \bbm 3\\2 \ebm,
&
M_2 &:= P(2) \oplus S(3) = \bbm 2\\1 \ebm \oplus \bbm 3 \ebm,
\\
M_3 &:= S(1) \oplus S(2) \oplus S(3) = \bbm 1 \ebm \oplus \bbm 2 \ebm \oplus \bbm 3 \ebm.
\end{align*}
These are the only modules with dimension vector $\bd$, up to isomorphism.
For $i = 1,2$ let $\calZ_i := \overline{\cO_{M_i}}$.
These are the only irreducible components of $\md(A,\bd)$.
It follows from \cite{DS81} that the scheme $\Spec(A,\bd)$
is reduced, hence
all modules in $\md(A,\bd)$ are reduced.
In particular,
both components $\calZ_1$ and $\calZ_2$ are generically reduced
(and generically regular).
$$
\begin{tabular}{r|c|c|c}
& $M_1$
& $M_2$
& $M_3$
\\\hline
reduced & yes & yes & yes
\\
regular & yes & yes & no
\\
$\tau$-regular & yes & no & no
\\
$\tau^-$-regular & no & yes & no
\\
$c(M)$ & 0 & 0 & 1
\\
$e(M)$ & 0 & 0 & 2
\\
$E(M)$ & 0 & 1 & 2
\\
$E^-(M)$ & 1 & 0 & 2
\end{tabular}
$$
We have $\irr^\tau(A,\bd) = \{ \calZ_1 \}$ and
$\irr^{\tau^-}(A,\bd) = \{ \calZ_2 \}$.


\section{Plamondon's construction of generically $\tau$-regular
components}\label{sec:taureg}


\subsection{Plamondon's theorems}\label{subsec:plamondon}
We recall Plamondon's construction and parametrization
of generically $\tau$-regular components.

Plamondon's \cite{P13} original proof is quite convoluted.
A more straightforward proof is due to Fei \cite{F23}.
It is based on a bundle construction from Derksen and Fei \cite{DF15}, see Section~\ref{subsec:bundle}.

For a given pair $(P_1,P_0)$ of projective $A$-modules
recall that
$$
r(P_1,P_0) := \max\{ \rk(f) \mid f \in \Hom_A(P_1,P_0) \}
$$
and
$$
\Hom_A(P_1,P_0)^\circ :=
\{ f \in \Hom_A(P_1,P_0) \mid \rk(f) = r(P_1,P_0) \}.
$$
By lower semicontinuity,
this is a dense open subset of the affine space $\Hom_A(P_1,P_0)$.
Let
$$
\bd_{P_1,P_0}^\circ := \dimv(\Coker(f))
$$
where $f \in \Hom_A(P_1,P_0)^\circ$.
(The cokernels of all such $f$ have the same dimension vector, see Lemma~\ref{lem:maximalrank}.)
Define
$$
\calZ_{P_1,P_0}^\circ :=
\{ M \in \md(A,\bd_{P_1,P_0}^\circ) \mid M \cong \Coker(f),\;
f \in \Hom_A(P_1,P_0)^\circ \},
$$
and let
$$
\calZ_{P_1,P_0} := \overline{\calZ_{P_1,P_0}^\circ}.
$$

\begin{Thm}[{Plamondon \cite[Theorem~1.2]{P13}}]\label{thm:plamondon1}
We have
$\calZ_{P_1,P_0} \in \irr^\tau(A)$, and each $\calZ \in \irr^\tau(A)$
is
of the form $\calZ_{P_1,P_0}$ for some pair $(P_1,P_0)$ of
projective $A$-modules.
\end{Thm}

Plamondon also explains when different pairs of projectives give rise
to the same component.
We reformulate this in terms of $g$-vectors.

For $\calZ \in \irr(A)$
let
$$
g^\circ(\calZ) := g(\calZ) + \zero(\calZ)
$$
where
$$
\zero(\calZ) := \sum_{\substack{1 \le i \le n\\i \notin \supp(\calZ)}} \N e_i.
$$

\begin{Thm}[{Plamondon \cite{P13}}]\label{thm:plamondon2}
We have
$$
\Z^n = \bigcup_{\calZ \in \irr^\tau(A)} g^\circ(\calZ)
$$
and this union is disjoint.
In particular, each $\calZ \in \irr^\tau(A)$ is determined by
$g(\calZ)$.
\end{Thm}

Theorem~\ref{thm:plamondon2} is not stated
explicitely in \cite{P13}, but it can be extracted from \cite[Theorem~1.2]{P13}.

A pair $(\calZ,P)$ with $\calZ \in \irr^\tau(A)$ and $P \in \proj(A)$ with
$\Hom_A(P,M) = 0$ for all $M \in \calZ$ is a \emph{$\tau$-regular pair}
for $A$.
(We consider $P$ up to isomorphism.)

Note that
for a fixed $\calZ$, the $\tau$-regular pairs $(\calZ,P)$ are parametrized
by the elements in $g^\circ(\calZ)$.

Here is a reformulation of Theorem~\ref{thm:plamondon2} in
terms of $\tau$-regular pairs:
There is a bijection between the set of $\tau$-regular pairs
and $\Z^n$ which sends such a pair $(\calZ,P)$ to $g(\calZ)+[P]$.

There is an obvious dual notion of \emph{$\tau^-$-regular pairs}, and
of course all statements can be dualized.

\subsection{Derksen and Fei's bundle construction}
\label{subsec:bundle}
The aim of this section is to outline the main ideas for the proofs
of Theorems~\ref{thm:plamondon1} and
\ref{thm:plamondon2}.
We use the simpler approach from \cite{DF15} and \cite{F23}.
But note that there is still a significant overlap with the more complicated proof by Plamondon \cite{P13}.
Both, \cite{P13} and \cite{DF15} assume that ${\rm char}(K) = 0$.
However, this is not necessary.
Furthermore, for the construction of generically $\tau$-regular components, \cite{F23} and \cite{P13} mainly consider
projective presentations
$P_1 \to P_0 \to M \to 0$
where $\add(P_1) \cap \add(P_0) = 0$.
(This can be assumed without loss of generality, if one just cares about
generic projective presentations.)
Since we want to look at all projective presentations of maximal rank,
we work without this assumption.

Fix a pair $(P_1,P_0)$ of projective $A$-modules.

The group
$$
G_{P_1,P_0} := \Aut_A(P_1) \times \Aut_A(P_0)
$$
acts on $\Hom_A(P_1,P_0)$ by
$$
(g_1,g_0).f := g_0fg_1^{-1}.
$$
For $f \in \Hom_A(P_1,P_0)$ let
$$
\cO_f := \{ (g_1,g_0).f \mid (g_1,g_0) \in G_{P_1,P_0} \}
$$
be the orbit of $f$.
The next lemma is straightforward.

\begin{Lem}
Let $f,f' \in \Hom_A(P_1,P_0)$.
Then the following are equivalent:
\begin{itemize}\itemsep2mm

\item[(i)]
$f' \in \cO_f$;

\item[(ii)]
$\Coker(f) \cong \Coker(f')$.

\end{itemize}
\end{Lem}

Let $\bd = (d_1,\ldots,d_n)$ be a dimension vector.
Let
$$
K^\bd := \prod_{i=1}^n K^{d_i}.
$$
This is the underlying vector space for the modules
in $\md(A,\bd)$.

Let
$$
\Hom_A(P_1,P_0,\bd) := \{ f \in \Hom_A(P_1,P_0) \mid
\dimv(\Coker(f)) = \bd \}.
$$
This is a locally closed subset of $\Hom_A(P_1,P_0)$.
We call $\Hom_A(P_1,P_0,\bd)$ a \emph{stratum} provided it is
non-empty.
Clearly, $\Hom_A(P_1,P_0)$ is the disjoint union of its strata.
The \emph{maximal stratum}
$$
\Hom_A(P_1,P_0)^\circ = \Hom_A(P_1,P_0,\bd_{P_1,P_0}^\circ)
$$
is dense and open in
$\Hom_A(P_1,P_0)$.
(Here $\bd_{P_1,P_0}^\circ$ is defined as in Section~\ref{subsec:plamondon}.)
Let
$Z(P_1,P_0,\bd)$
be the set of all triples
$$
(f,\pi,M) \in \Hom_A(P_1,P_0) \times \Hom_{Q_0}(P_0,K^\bd) \times
\md(A,\bd)
$$
such that
$$
P_1 \xrightarrow{f} P_0 \xrightarrow{\pi} K^\bd \to 0
$$
is an exact sequence, and $M$ is the unique $A$-module structure
on $K^\bd$ such that $\pi$ is an $A$-module homomorphism.
Here we have
$$
\Hom_{Q_0}(P_0,K^\bd) := \prod_{i=1}^n \Hom_K(e_iP_0,K^{d_i}).
$$

We obtain two maps of quasi-projective varieties
$$
\SelectTips{cm}{}
\xymatrix{
& Z(P_1,P_0,\bd) \ar[dl]_{q_1}\ar[dr]^{q_2}
\\
\Hom_A(P_1,P_0,\bd) && \md(A,\bd)
}
$$
defined by $q_1(f,\pi,M) := f$ and $q_2(f,\pi,M) := M$.
This diagram is defined and studied by Derksen and
Fei \cite[Section~2]{DF15}, see also \cite[Section~3.2]{F23}.

The group $G :=
G_{P_1,P_0} \times G_\bd$ acts on $Z(P_1,P_0,\bd)$ by
$$
((g_1,g_0),g).(f,\pi,M) := (g_0fg_1^{-1},g \pi g_0^{-1},g.M)
$$
where $g.M$ comes from the $G_\bd$-action on $\md(A,\bd)$.
Note that $g.M$ is
the unique $A$-module structure on $K^\bd$
such that $g \pi g_0^{-1}$ and $g$ are $A$-module
homomorphisms.

By restriction, the groups $G_{P_1,P_0}$ and $G_\bd$
also act
on $Z(P_1,P_0,\bd)$.
Namely, for
$(g_1,g_0) \in G_{P_1,P_0}$, $g \in G_\bd$ and
$(f,\pi,M) \in Z(P_1,P_0,\bd)$ we have
$$
(g_1,g_0).(f,\pi,M) := (g_0fg_1^{-1},\pi g_0^{-1},M)
\text{\quad and \quad}
g.(f,\pi,M) := (f,g\pi,g.M).
$$

\begin{Lem}\label{lem:orbitbij1}
The map $q_1$ is $G_{P_1,P_0}$-equivariant, and
$q_2$ is $G_\bd$-equivariant.
The maps $q_1$ and $q_2$ induce a bijection
$$
\SelectTips{cm}{}
\xymatrix{
\{ \text{$G_{P_1,P_0}$-orbits in $\Hom_A(P_1,P_0,\bd)$} \} \ar[r]^<<<<{q_2q_1^{-1}}
& \{ \text{$G_\bd$-orbits in $\Ima(q_2)$} \}.
}
$$
\end{Lem}

\begin{Lem}[{\cite[Lemma~2.4]{DF15}}]
The map $q_1$ is a principal $G_\bd$-bundle.
In particular,
$$
q_1^{-1}(f) \cong G_\bd
$$
for all
$f \in \Hom_A(P_1,P_0,\bd)$.
\end{Lem}

It is claimed in \cite[Theorem~2.3]{DF15} that
$q_2$ is an open map (i.e.\ that it maps
open subsets to open subsets).
But the proof in \cite{DF15} does not seem to work.
However, Fei \cite{F23b} found an argument which shows
the following weaker (but still good enough) result:

\begin{Lem}[{Fei \cite{F23b}}]\label{lem:DFopen}
The image of $q_2$ is open.
\end{Lem}

\begin{proof}
Let $M \in \Ima(q_2)$.
Then $M \in \calZ$ for some $\calZ \in \irr(A,\bd)$.
It follows from Fei's
Lemma~\ref{lem:presentationsemicont} that
there is a dense open subset $\calU_M$ of $\calZ$
such that $M \in \calU_M$ and
$\calU_M \subseteq \Ima(q_2)$.
We have
$$
\Ima(q_2) = \bigcup_{M \in \Ima(q_2)} \calU_M
$$
and this union of  open sets is of course again open.
\end{proof}

The following lemma is not stated explicitely in \cite{DF15} nor \cite{F23}.
Thus we add a proof.

\begin{Lem}
Let $M \in \Ima(q_2)$.
Then for each $(f,\pi,M) \in Z(P_1,P_0,\bd)$
we have
$$
q_2^{-1}(M) \cong \cO_f \times \Aut_A(M).
$$
\end{Lem}

\begin{proof}
The group $G_M := G_{P_1,P_0} \times \Aut_A(M)$ acts on the fibre
$q_2^{-1}(M)$.
Namely, for $(f,\pi,M) \in q_2^{-1}(M)$ and
$((g_1,g_0),g) \in G_M$ we have
$((g_1,g_0),g).(f,\pi,M) := (f',\pi',M')$ where
$f' = g_0fg_1^{-1}$, $\pi' = g \pi g_0^{-1}$
and $M'$ is the unique $A$-module structure on
$K^\bd$ given by $g \pi g_0^{-1}$.
Since $g \in \Aut_A(M)$ we get $M' = M$ and
therefore $(f',\pi',M') \in q_2^{-1}(M)$.

This action is transitive:
For $(f',\pi',M),(f,\pi,M) \in q_2^{-1}(M)$ there exists
some $(g_1,g_0) \in G_{P_1,P_0}$ such that
$f'g_1 = g_0f$ and $\pi'g_0 = \pi$.
Thus $((g_1,g_0),1_M).(f,\pi,M) = (f',\pi',M)$.
$$
\SelectTips{cm}{}
\xymatrix{
P_1 \ar[r]^f\ar@{-->}[d]^{g_1} &
P_0 \ar[r]^{\pi}\ar@{-->}[d]^{g_0} &
M \ar[r]\ar[d]^{1_M} & 0
\\
P_1 \ar[r]^{f'} & P_0 \ar[r]^{\pi'} & M \ar[r] & 0
}
$$

Let $G_{(f,\pi,M)}$ be the stabilizer of
$(f,\pi,M) \in Z(P_1,P_0,\bd)$ for the $G_M$-action on $q_2^{-1}(M)$, and let
$G_f$ be the stabilizer of $f$ for the
$G_{P_1,P_0}$-action on $\Hom_A(P_1,P_0,\bd)$.

For $((g_1,g_0),g) \in G_{(f,\pi,M)}$
we have $(g_1,g_0) \in G_f$, and
the element $g$ is uniquely determined by $(g_1,g_0)$.
Vice versa, for each $(g_1,g_0) \in G_f$
there is a unique $g \in \Aut_A(M)$ such that
$((g_1,g_0),g) \in G_{(f,\pi,M)}$.
We get a group isomorphism
$G_{(f,\pi,M)} \to G_f$ which maps $((g_1,g_0),g)$
to $(g_1,g_0)$.
$$
\SelectTips{cm}{}
\xymatrix{
P_1 \ar[d]^{g_1}\ar[r]^{f} & P_0 \ar[d]^{g_0}\ar[r]^\pi
& M \ar[r]\ar@{-->}[d]^{g} & 0
\\
P_1 \ar[r]^{f} & P_0 \ar[r]^\pi
& M \ar[r] & 0
}
$$
It follows that
$$
q_2^{-1}(M) \cong G_M/G_{(f,\pi,M)} \cong
(G_{P_1,P_0}/G_f) \times \Aut_A(M) \cong \cO_f \times
\Aut_A(M).
$$
\end{proof}

As before,
for $f \in \Hom_A(P_1,P_0)$ let
$P_f$ be a maximal direct summand of $P_1$
such that $f(P_f) = 0$.

The following lemma can be derived from
\cite[Lemmas~2.6 and 2.16]{P13}.
It is non-trivial and needs some work.

\begin{Lem}
Let $(f,\pi,M) \in Z(P_1,P_0,\bd)$.
Then we have
$$
\dim \cO_f = \hom_A(P_1,P_0) - \hom_A(M,\tau_A(M)) - \hom_A(P_f,M).
$$
\end{Lem}

Let $\cC$ be a closed irreducible stable subset
of $\Hom_A(P_1,P_0,\bd)$.
Let
$$
\eta_\bd(\cC) := \overline{q_2(q_1^{-1}(\cC))}
$$
where the closure is taken in $\md(A,\bd)$.

The next three lemmas follow from the already stated properties
of $q_1$ and $q_2$.

\begin{Lem}
$\eta_\bd(\cC)$ is a closed irreducible stable subset of
$\overline{\Ima(q_2)}$.
\end{Lem}

\begin{Lem}
For each closed irreducible stable subset $\calZ$ of
$\overline{\Ima(q_2)}$ there is a
closed irreducible stable subset of $\Hom_A(P_1,P_0,\bd)$
such that $\eta_\bd(\cC) = \calZ$.
\end{Lem}

Let $\irr(\Hom_A(P_1,P_0,\bd))$ denote the set of irreducible components
of the stratum $\Hom_A(P_1,P_0,\bd)$.

\begin{Lem}\label{lem:imageq2}
Let $\calZ \in \irr(A,\bd)$.
Then there is some
$\cC \in \irr(\Hom_A(P_1^{\calZ},P_0^{\calZ},\bd))$
with $\eta_\bd(\cC) = \calZ$.
\end{Lem}

Also
the following theorem can be derived without difficulty from the stated properties of $q_1$ and $q_2$ together with Lemma~\ref{lem:chevalley}.

\begin{Thm}\label{thm:bundle1}
The following hold:
\begin{itemize}\itemsep2mm

\item[(i)]
Let $\cC$ be a closed irreducible stable subset of
$\Hom_A(P_1,P_0,\bd)$, and let $\calZ := \eta_\bd(\cC)$.
Then
$$
c(\calZ) = E(\calZ) + \hom_A(P_f,M) - {\rm codim}(\cC)
$$
where $f$ is generic in $\cC$ and $M$ is generic in $\calZ$, and
$$
{\rm codim}(\cC) := \hom_A(P_1,P_0) - \dim(\cC).
$$
Furthermore,
$\hom_A(P_f,M) - {\rm codim}(\cC) \le 0$.

\item[(ii)]
If $\cC := \Hom_A(P_1,P_0,\bd)$ is irreducible, then
$\calZ := \eta_\bd(\cC) \in \irr(A,\bd)$.

\item[(iii)]
The stratum
$\cC := \Hom_A(P_1,P_0)^\circ$ is irreducible and
$$
\calZ := \eta_\bd(\cC) \in \irr^\tau(A,\bd_{P_1,P_0}^\circ).
$$

\end{itemize}
\end{Thm}

Let $\calZ_{P_1,P_0}$ be defined as in
Section~\ref{subsec:plamondon}.
Then $\calZ_{P_1,P_0}$ coincides with the
generically $\tau$-regular component $\calZ$ appearing
in Theorem~\ref{thm:bundle1}(iii).

\begin{Cor}\label{cor:bundle2}
Let $\calZ \in \irr(A,\bd)$, and let
$\cC \in \irr(\Hom_A(P_1^{\calZ},P_0^{\calZ},\bd))$
with $\eta_\bd(\cC) = \calZ$.
Then the following hold:
\begin{itemize}\itemsep2mm

\item[(i)]
$P_f = 0$ with $f$ generic in $\cC$;

\item[(ii)]
The following are equivalent:
\begin{itemize}

\item[(a)]
$c(\calZ) = E(\calZ)$;

\item[(b)]
$\bd = \bd_{P_1^{\calZ},P_0^{\calZ}}^\circ$;

\item[(c)]
$\calZ = \calZ_{P_1^{\calZ},P_0^{\calZ}}$;

\item[(d)]
${\rm codim}(\cC) = 0$.

\end{itemize}
In this case, $\cC = \Hom_A(P_1^{\calZ},P_0^{\calZ},\bd)$.

\end{itemize}
\end{Cor}

\begin{proof}
(i):
A generic $f$ in $\cC$ gives a minimal projective presentation
of $\Coker(f) \in \calZ$.
Thus $P_f = 0$.

(ii):
(a) $\implies$ (d):
It follows from (i) combined with Theorem~\ref{thm:bundle1}(i)
that ${\rm codim}(\cC) = 0$.

(d) $\implies$ (b):
Condition (d) implies that
$\cC = \Hom_A(P_1^\calZ,P_0^\calZ,\bd) = \Hom_A(P_1^\calZ,P_0^\calZ)^\circ$ and therefore
(b) holds.

(b) $\implies$ (c):
Condition (b) implies that
$\cC = \Hom_A(P_1^\calZ,P_0^\calZ)^\circ$.
This implies (c).

(c) $\implies$ (d):
Since
$\calZ = \calZ_{P_1^\calZ,P_0^\calZ}$, we have $\bd = \bd_{P_1^\calZ,P_0^\calZ}^\circ$.
Thus we have $\cC = \Hom_A(P_1^\calZ,P_0^\calZ)^\circ$ and therefore ${\rm codim}(\cC) = 0$.

(d) $\implies$ (a):
The condition ${\rm codim}(\cC) = 0$ and the inquality in
Theorem~\ref{thm:bundle1}(i) imply that $c(\calZ) = E(\calZ)$.
\end{proof}

Note that Theorem~\ref{thm:bundle1} combined with
Corollary~\ref{cor:bundle2} implies Plamondon's Theorem~\ref{thm:plamondon1}.

\begin{Cor}\label{cor:bundle3}
Let $\calZ \in \irr^\tau(A)$, and let
$f \in \Hom_A(P_1^{\calZ},P_0^{\calZ})^\circ$.
Then the following hold:
\begin{itemize}\itemsep2mm

\item[(i)]
$P_1^{\calZ} \xrightarrow{f} P_0^{\calZ} \to \Coker(f) \to 0$
is a minimal projective presentation;

\item[(ii)]
$f(P) \not= 0$ for all non-zero direct summands $P$ of $P_1^{\calZ}$.

\end{itemize}
\end{Cor}

\begin{proof}
(i):
Set $M := \Coker(f)$.
We know from Corollary~\ref{cor:bundle2}(ii) that $\calZ = \calZ_{P_1^{\calZ},P_0^{\calZ}}$, and that
$M \in \calZ$.
Suppose that the projective presentation
$$
P_1^{\calZ} \xrightarrow{f} P_0^{\calZ} \to M \to 0
$$
is not minimal.
Then the minimal projective presentation of $M$ is a proper
direct summand of this presentation, see
Lemma~\ref{lem:gvector1}.
Now Lemma~\ref{lem:presentationsemicont} yields a contradiction.

(ii):
This follows directly from (i).
\end{proof}

\begin{Cor}[{\cite[Corollary~2.17]{P13}}]\label{cor:bundle4}
Let $\calZ \in \irr^\tau(A,\bd)$.
Then
$$
\add(P_1^{\calZ}) \cap \add(P_0^{\calZ}) = 0.
$$
\end{Cor}

\begin{proof}
Let $(P_1,P_0) := (P_1^\calZ,P_0^\calZ)$.
By Corollary~\ref{cor:bundle2}
we have $\eta_\bd(\cC) = \calZ = \calZ_{P_1,P_0}$ where
$\cC := \Hom_A(P_1,P_0)^\circ$.

Let $\pi\df P_0 \to \tp(P_0)$ be the obvious projection
with $\Ker(\pi) = \rad(P_0)$.
By Corollary~\ref{cor:bundle3},
each $f \in \cC$ yields a minimal projective presentation
of $\Coker(f)$.
In particular, we have $\Ima(f) \subseteq \rad(P_0)$ and
therefore $\pi f = 0$.

Now suppose that $P_1$ and $P_0$ have a common non-zero
direct summand $P$.
Then there must be some $f \in \Hom_A(P_1,P_0)$ with
$\pi f \not= 0$.

Since the map
\begin{align*}
\Hom_A(P_1,P_0) &\to \Z
\\
f &\mapsto \rk(\pi f)
\end{align*}
is lower semicontinuous, there must be some $f \in \cC$ with
$\pi f \not= 0$, a contradiction.
\end{proof}

\begin{Cor}\label{cor:bundle8}
Let $\calZ = \calZ_{P_1,P_0}$, and let
$P \in \proj(A)$ such that
$\Hom_A(P,M) = 0$ for all $M \in \calZ$.
Then
$$
\calZ_{P_1,P_0} = \calZ_{P_1 \oplus P,P_0}.
$$
\end{Cor}

\begin{proof}
Let $\cC := \Hom_A(P_1 \oplus P,P_0)^\circ$.
For generic $f = (f_1,f_2) \in \cC$ we have
$f_1 \in \Hom_A(P_1,P_0)^\circ$ and therefore
$M := \Coker(f_1) \in \calZ$.
Since $\Hom_A(P,M) = 0$, we get $\Ima(f_2) \subseteq \Ima(f_1)$.
In other words, $\Coker(f) = M$.
This yields the result.
\end{proof}

\begin{Cor}\label{cor:bundle5}
Assume that $P$ is a common direct summand
of $P_1$ and $P_0$.
Then
$$
\calZ_{P_1,P_0} = \calZ_{P_1/P,P_0/P}.
$$
\end{Cor}

\begin{proof}
Let $f$ be generic in $\Hom_A(P_1,P_0)$.
It follows from \cite[Corollary~4.2]{DF15} that $f$ has
a direct summand isomorphic to $1_P\df P \to P$.
This yields the result.
\end{proof}

\begin{Lem}\label{lem:gvector2}
\label{prop:plamondon1b}
Let $\calZ := \calZ_{P_1,P_0}$, and let
$f$ be generic in $\Hom_A(P_1,P_0)$.
Then
$$
g(\calZ) = [P_1^{\calZ}] - [P_0^{\calZ}] = [P_1] - [P_0] - [P_f]
$$
and $\Hom_A(P_f,M) = 0$ for all $M \in \calZ$.
\end{Lem}

\begin{proof}
For generic $f$ in
$\Hom_A(P_1,P_0)$ we have
$\Hom_A(P_f,M) = 0$ for all $M \in \calZ$.
(This follows from Lemma~\ref{lem:maxrankfM}.)
Now the claim follows from Lemma~\ref{lem:gvector1}.
\end{proof}

\begin{Lem}\label{lem:bundle6}
Assume that $\cC := \Hom_A(P_1,P_0,\bd)$ is an irreducible stratum of
$\Hom_A(P_1,P_0)$, and let
$\calZ := \eta_\bd(\cC)$.
Then for each $f \in \cC$ we have $\Coker(f) \in \calZ^{\rm int}$.
\end{Lem}

\begin{proof}
The properties of $q_1$ and $q_2$ imply that
$\calZ \in \irr(A,\bd)$.
Let $f \in \cC$.
Assume that $\Coker(f) \in \calZ'$ for some $\calZ' \in \irr(A,\bd)$.
Then Fei's Lemma~\ref{lem:presentationsemicont} says that there
is a dense open subset $\calU'$ of $\calZ'$ such that
each $M' \in \calU'$ has a projective presentation
of the form
$$
P_1 \to P_0 \to M' \to 0.
$$
It follows that there is some $\cC' \in \irr(\Hom_A(P_1,P_0,\bd))$
with $\eta_\bd(\cC') = \calZ'$.
Since $\Hom_A(P_1,P_0,\bd)$ is irreducible, we get
$\cC' = \cC$ and therefore
$\calZ' = \calZ$.
\end{proof}

\begin{Cor}\label{cor:bundle7}
Let $\calZ := \calZ_{P_1,P_0}$.
Then for each $f \in \Hom_A(P_1,P_0)^\circ$ we have
$\Coker(f) \in \calZ^{\rm int}$.
\end{Cor}

\begin{proof}[{Proof of Plamondon's Theorem~\ref{thm:plamondon2}}]
We can write each $z \in K_0(\proj(A))$ uniquely as
$z = [P_1] - [P_0]$ for some pair $(P_1,P_0)$ of projective $A$-modules with $\add(P_1) \cap \add(P_0) = 0$.
Set $(P_1^z,P_0^z) := (P_1,P_0)$.
Due to Theorem~\ref{thm:plamondon1} combined with Corollary~\ref{cor:bundle4}
we have a surjective map
\begin{align*}
\eta\df K_0(\proj(A)) &\to \irr^\tau(A)
\\
z &\mapsto \calZ_{P_1^z,P_0^z}.
\end{align*}

Let $\calZ \in \irr^\tau(A)$, and let $z \in K_0(\proj(A))$
with $\eta(z) = \calZ$.
We know that
$$
g(\calZ) = [P_1^{\calZ}] - [P_0^{\calZ}] =
[P_1^z] - [P_0^z] - [P_f]
$$
with $\Hom_A(P_f,M) = 0$ for all $M \in \calZ$, where
$f$ is generic in $\Hom_A(P_1^z,P_0^z)$.
In particular, we have
$[P_f] \in \zero(\calZ)$.
Thus $z  = [P_1^z] - [P_0^z] = g(\calZ) + [P_f] \in g^\circ(\calZ)$.

For $z' = g^\circ(\calZ)$ we have
$z' = [P_1^\calZ] - [P_0^\calZ] + [P]$ for some
$P \in\proj(A)$ with $\Hom_A(P,M) = 0$ for all $M \in \calZ$.
We get
$$
\calZ = \calZ_{P_1^\calZ,P_0^\calZ} = \calZ_{P_1^\calZ \oplus P,P_0^\calZ} = \eta(z').
$$
(The second equality follows from Corollary~\ref{cor:bundle8}, and for the third equality we use that $\add(P_1^\calZ \oplus P) \cap \add(P_0^\calZ) = 0$.)

This finishes the proof.
\end{proof}

\subsection{Examples}

\subsubsection{}
It can happen that a stratum $\cC := \Hom_A(P_1,P_0,\bd)$ is
irreducible with $\bd \not= \bd_{P_1,P_0}^\circ$ such
that $\eta_\bd(\cC) \in \irr^\tau(A,\bd)$.
As  a trivial example, assume that $\Hom_A(P_1,P_0) \not= 0$, and let $\bd := \dimv(P_0)$.
Then $\cC$ consists just of the zero map, and we get $\eta_\bd(\cC) = \overline{\cO_{P_0}}$.

\subsubsection{}
Let $A = KQ/I$ where $Q$ is the quiver
$$
\xymatrix@-1ex{
3 \ar[d] & 4 \ar[d]
\\
1 \ar@/^1ex/[r]^a& 2 \ar@/^1ex/[l]^b
}
$$
and $I = (ab,ba)$.
Let
$$
P_1 := P(1) \oplus P(2) =
\bbm 1\\2\ebm \oplus
\bbm 2\\1\ebm
\text{\quad and \quad}
P_0 := P(3) \oplus P(4) =
\bbm 3\\1\\2 \ebm \oplus \bbm 4\\2\\1 \ebm.
$$
For $\bd = (1,1,1,1)$ the stratum $\Hom_A(P_1,P_0,\bd)$ has two irreducible
components, say $\cC_1$ and $\cC_2$.
Both $\calZ_1 := \eta_\bd(\cC_1)$ and $\calZ_2 := \eta_\bd(\cC_2)$
are generically $\tau$-regular irreducible components.
Up to renumbering, we have
$\calZ_1 = \calZ_{P(1),P(3) \oplus P(4)}$ and
$\calZ_2 = \calZ_{P(2),P(3) \oplus P(4)}$.
Note that
$\calZ_{P(1) \oplus P(2),P(3) \oplus P(4)} = S(3) \oplus S(4)
= \md(A,(0,0,1,1))$.

\subsubsection{}
Let $A = KQ/I$ where $Q$ is the quiver
$$
\xymatrix@-1ex{
1 \ar@/^1ex/[r]^a& 2 \ar@/^1ex/[l]^b
}
$$
and $I$ is generated by $aba$.
Then we have
$$
P(1) = \bbm 1\\2\\1\ebm
\text{\quad and \quad}
P(2) = \bbm 2\\1\\2\\1\ebm.
$$
For $P_1 := P(1)$ and $P_0 := P(2)$, $\Hom_A(P_1,P_0)$ has three strata, namely for
$\bd \in \{ (2,2),\; (1,2),\; (0,1) \}$.
All are irreducible, and the ones for $(2,2)$ and $(0,1)$
give rise to the generically $\tau$-regular
components $\calZ_{0,P(2)}$ and $\calZ_{P(1),P(2)}$, respectively.


\section{Projective presentations of maximal rank and
$\tau$-regular modules}\label{sec:maximalrank}


\subsection{Characterization $\tau$-regular modules}
The following theorem characterizes $\tau$-regular modules in terms
of their projective presentations.
As already mentioned in the introduction, it can be seen as a refinement of aspects of Plamondon's Theorem~\ref{thm:plamondon1}.

\begin{Thm}\label{thm:main0b}
For $M \in \md(A)$ let
\begin{equation}\label{eq:presentation}
P_1 \xrightarrow{f} P_0 \to M \to 0
\end{equation}
be a projective presentation.
Then
the following hold:
\begin{itemize}\itemsep2mm

\item[(i)]
If $\rk(f) = r(P_1,P_0)$, then $M$ is $\tau$-regular;

\item[(ii)]
If $M$ is $\tau$-regular and if (\ref{eq:presentation}) is a minimal presentation,
then $\rk(f) = r(P_1,P_0)$.

\end{itemize}
\end{Thm}

\begin{proof}
Let $\bd := \dimv(M)$.

(i):
Let $\calZ := \calZ_{P_1,P_0}$.
We assume that $\rk(f) = r(P_1,P_0)$,
i.e.\ $\bd = \bd_{P_1,P_0}^\circ$.
Thus $\cC := \Hom_A(P_1,P_0,\bd) = \Hom_A(P_1,P_0)^\circ$ is irreducible,
$\calZ = \eta_\bd(\cC)$,
and Lemma~\ref{lem:bundle6} implies that
$M \in \calZ^{\rm int}$.
Since $\rk(f) = r(P_1,P_0)$, Lemma~\ref{lem:maxrankfM}
says that $\Hom_A(P_f,M) = 0$.
By Lemma~\ref{lem:gvector3}
there is an exact sequence
$$
0 \to \Hom_A(M,M) \to \Hom_A(P_0,M) \to \Hom_A(P_1,M)
\to \Hom_A(M,\tau_A(M)) \to 0.
$$
The alternating sum of the dimensions of these
four spaces is zero.
We have
$$
c(M) = \dim(\calZ) - \dim \cO_M = \dim(\calZ) - \dim G_\bd + \hom_A(M,M).
$$
Recall that $c(M) \le E(M) = \hom_A(M,\tau_A(M))$.

Now an easy calculation yields that
$c(M) = E(M)$ if and only if
$$
\dim(\calZ) - \dim G_\bd = -\hom_A(P_0,M) + \hom_A(P_1,M).
$$
Note that
these numbers only depend on $\dimv(M)$.

Since $\calZ$ is generically $\tau$-regular, there exists some
$N \in \calZ$ such that $c(N) = E(N)$ and $N$ has
a projective presentation
$$
P_1 \xrightarrow{g} P_0 \to N \to 0.
$$
(For the latter we used Fei's Lemma~\ref{lem:presentationsemicont}.)
Again by Lemma~\ref{lem:gvector3}
we get an exact sequence
\begin{multline*}
0 \to \Hom_A(N,N) \to \Hom_A(P_0,N) \to \Hom_A(P_1,N)
\\
\to \Hom_A(N,\tau_A(N)) \oplus \Hom_A(P_g,N) \to 0.
\end{multline*}
Of course, we have $\dimv(M) = \dimv(N)$ and therefore
$\rk(g) = r(P_1,P_0)$.
This implies $\Hom_A(P_g,N) = 0$, see again Lemma~\ref{lem:maxrankfM}.
Now $c(N) = E(N)$ together with the considerations above implies that
$c(M) = E(M)$.

(ii):
Assume that $M$ is $\tau$-regular, i.e.\ $c(M) = E(M)$
and that the presentation (\ref{eq:presentation}) is minimal.
By Lemma~\ref{lem:regular}
there is a unique $\calZ \in \irr(A,\bd)$ with $M \in \calZ$, and by
upper semicontinuity, $\calZ$ is generically $\tau$-regular.
For $N$ generic in $\calZ$ we have $c(N) = E(N)$ and there is a projective presentation
$$
P_1 \xrightarrow{g} P_0 \to N \to 0.
$$
(Here we used again Lemma~\ref{lem:presentationsemicont}.)
We have
$$
g(N) = [P_1] - [P_0] - [P_g].
$$
For all $L \in \calZ^{\rm int}$ we have
\begin{align*}
\bil{g(L),[L]} &= - \hom_A(L,L) + \hom_A(L,\tau_A(L))
\\
&= E(L) - c(L) + \dim(\calZ) - \dim G_\bd.
\end{align*}
Now $c(M) = E(M)$ and $c(N) = E(N)$ imply that
$$
\bil{g(M),[M]} = \dim(\calZ) - \dim G_\bd = \bil{g(N),[N]}.
$$
We get
$$
\bil{[P_1]-[P_0],[M]} = \bil{g(M),[M]} = \bil{g(N),[N]} =
\bil{[P_1]-[P_0],[N]} - \bil{[P_g],[N]}.
$$
Since $[M] = [N] = \bd$, this
implies $\hom_A(P_g,N) = \bil{[P_g],[N]} = 0$.

Let $\cC \in \irr(\Hom_A(P_1,P_0,\bd))$ such that
$\eta_\bd(\cC) = \calZ$.
We can assume that $g$ is generic in $\cC$ and $N$ is generic in
$\calZ$.
Now Theorem~\ref{thm:bundle1}(i) says that
$$
c(\calZ) = E(\calZ) + \hom_A(P_g,N) - {\rm codim}(\cC)
= E(\calZ) - {\rm codim}(\cC).
$$
Since $c(\calZ) = E(\calZ)$,
this implies that ${\rm codim}(\cC) = 0$, and therefore
$\bd = \bd_{P_1,P_0}^\circ$.
In other words, $\rk(g) = r(P_1,P_0)$.
Since $\dimv(M) = \dimv(N)$ we also get $\rk(f) = r(P_1,P_0)$.
\end{proof}

\subsection{Examples}

\subsubsection{}
Let $A = KQ$ where $Q$ is the quiver
$$
\xymatrix{
1 & 2 \ar[l]
}
$$
Let
$$
P_1 := P(1) = \bbm 1 \ebm
\text{\quad and \quad}
P_0 := P(1) \oplus P(2) = \bbm 1 \ebm \oplus \bbm 2\\1 \ebm.
$$
Clearly, each non-zero homomorphism $f\df P_1 \to P_0$ is
a monomorphism and therefore of maximal rank $r(P_1,P_0) = 1$.
Up to a non-zero scalar, exactly one of these, say $f_0$, has a cokernel isomorphic to
$M = S(1) \oplus S(2) = \bbm 1 \ebm \oplus \bbm 2 \ebm$.
In all other cases, we have $\Coker(f) \cong P(2)$.
The map $f_0$ is not a general projective presentation in the sense of \cite{DF15}, but for
its cokernel $M$ we still have $c(M) = E(M)$.

\subsubsection{}
Let $A = KQ/I$ where $Q$ is the quiver
$$
\xymatrix{
1 \ar@(ul,dl)[]_a & 2 \ar[l]
}
$$
and $I$ is generated by $a^2$.
The Auslander-Reiten quiver of $A$ looks as follows:
$$
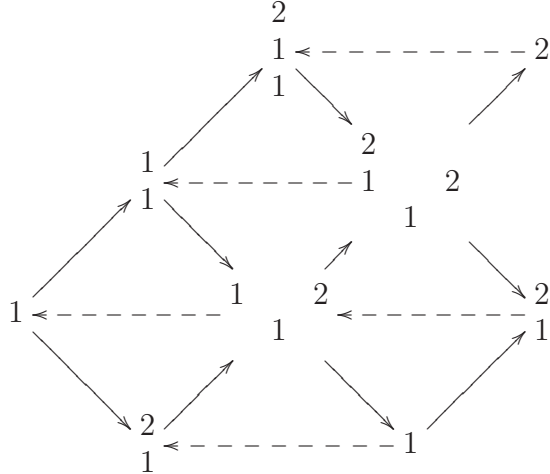
\xymatrix@!@-4.5ex{
&& {\bbm 2\\1\\1 \ebm}\ar[dr] &&  {\bbm 2 \ebm}
\ar@{-->}[ll]
\\
& {\bbm 1\\1 \ebm}\ar[dr]\ar[ur]  &&
{\bbm 2\\1&&2\\&1 \ebm} \ar@{-->}[ll]\ar[dr]\ar[ur]
\\
 {\bbm 1 \ebm}\ar[ur]\ar[dr] &&
{\bbm 1&&2\\&1 \ebm}\ar[dr]\ar[ur]\ar@{-->}[ll]  &&
{\bbm 2\\1 \ebm}\ar@{-->}[ll]
\\
&  {\bbm 2\\1 \ebm}\ar[ur] && {\bbm 1 \ebm}\ar@{-->}[ll]  \ar[ur]
}
$$
(One needs to identify the two modules on the left of the
3rd and 4th row with the two modules on the right of the
4th and 3rd row, respectively.)
For $\bd = (2,1)$ we have $\calZ := \overline{\cO_{P(2)}} = \md(A,\bd)$,
i.e.\ $\md(A,\bd)$ is irreducible.
We have $\dim(\calZ) = 4$.
The modules with dimension vector $\bd$ are
\begin{align*}
M_1 &= \bbm 2\\1\\1 \ebm,
&
M_2 &= \bbm 1&&2\\&1\ebm,
&
M_3 &= \bbm 1\\1 \ebm \oplus \bbm 2 \ebm,
\\
M_4 &= \bbm 1 \ebm \oplus \bbm 2\\1\ebm,
&
M_5 &= \bbm 1 \ebm \oplus \bbm 1 \ebm \oplus \bbm 2 \ebm.
\end{align*}
Note that $P(2) = M_1$.
The minimal projective presentations of the modules $M_i$ are of the form
\begin{align*}
0 \to P(2) \to M_1 \to 0,
\\
P(1) \to P(1) \oplus P(2) \to M_2 \to 0,
\\
P(1) \to P(1) \oplus P(2) \to M_3 \to 0,
\\
P(1) \oplus P(1) \to P(1) \oplus P(2) \to M_4 \to 0,
\\
P(1) \oplus P(1) \oplus P(1) \to
P(1) \oplus P(1) \oplus P(2) \to M_5 \to 0.
\end{align*}
We have
$$
\begin{tabular}{r|c|c|c|c|c}
& $M_1$ & $M_2$ & $M_3$ & $M_4$ & $M_5$
\\\hline
$c(M)$ & 0 & 1 & 2 & 2 & 4
\\
$E(M)$ & 0 & 1 & 2 & 4 & 6
\end{tabular}
$$
Now Theorem~\ref{thm:main0b} says that
the modules $M_1$, $M_2$ and $M_3$ have a minimal projective presentation of maximal rank whereas $M_4$ and $M_5$ do not.
The only module which has a general projective presentation in the sense of \cite{DF15} is $M_1$.


\section{Reduction technique for $\tau$-regular modules}\label{sec:reduction}


Let $I$ be an ideal in our algebra $A$, and let $B := A/I$.
Then we interpret $\md(B)$ as the subcategory
of $\md(A)$ consisting of all $A$-modules $M$ with
$IM = 0$.
In this case, we have
$$
\{ \calZ \in \irr(A) \mid I\calZ = 0 \} \subseteq \irr(B).
$$

Let now $I = AeA$ for some idempotent $e$ in $A$.
We have $Ae \in \proj(A)$, and for $M \in \md(A)$
there is an isomorphism
$\Hom_A(Ae,M) \cong eM$ of $K$-vector spaces.
Then
$\md(B)$ consists of
all $A$-modules $M$ such that
$$
\supp(M) \subseteq \supp(B) :=
\{ 1 \le i \le n \mid eS(i) = 0 \}.
$$
It follows that
$$
\{ \calZ \in \irr(A) \mid I\calZ = 0 \} = \irr(B).
$$
In particular, we get $\irr(B) \subseteq \irr(A)$.
Furthermore, we have
$$
\md_M(B,\bd) = \md_M(A,\bd)
$$
for all $M \in \md(B,\bd)$.

The following proposition is part of \cite[Proposition~5.2]{MP23}.
We include a short proof.

\begin{Prop}[{Mousavand, Paquette}]\label{prop:MP23}
Let $I$ be an ideal in $A$, and let $B:= A/I$.
Let $\calZ \in \irr^\tau(A)$ such that $I\calZ = 0$.
Then $\calZ \in \irr^\tau(B)$.
\end{Prop}

\begin{proof}
Note that $c_B(\calZ) = c_A(\calZ)$.

Assume that $\calZ \notin \irr^\tau(B)$.
By definition it means that for all $M \in \calZ$ we have
$$
\hom_A(M,\tau_B(M)) = \hom_B(M,\tau_B(M)) = E_B(M)
> c_B(\calZ) = c_A(\calZ).
$$
By \cite[Proposition~4.2]{AR77} we know that $\tau_B(M)$ is isomorphic
to a submodule of $\tau_A(M)$.
Consequently,
\[
E_A(M) = \hom_A(M,\tau_A(M)) \geq \hom_A(M,\tau_B(M)) > c_A(\calZ)
\]
for all $M \in \calZ$.
Thus $\calZ \notin \irr^\tau(A)$, a contradiction.
\end{proof}

\noindent
{\bf Examples}:
\begin{itemize}\itemsep2mm

\item[(i)]
The converse of
Proposition~\ref{prop:MP23} is in general wrong.
For example, let $A = KQ/I$ where
$Q$ is the quiver
$$
\xymatrix{1 \ar@(ur,dr)[]^a}
$$
and $I = (a^3)$.
Furthermore, let $B = KQ/(a^2) = A/(a^2)$.
Let
$$
M := {_B}B = \bbm 1\\1 \ebm,
$$
and let $\calZ := \overline{\cO_M}$.
Then $\calZ \in \irr^\tau(B)$, since $M$ is a projective $B$-module.
We have $\calZ \in \irr(A)$, but $\calZ \notin \irr^\tau(A)$.
Note that $M$ is $\tau_B$-regular, but not $\tau_A$-regular.
(We have $c_B(M) = E_B(M) = 0$, $c_A(M) = 0$ and $E_A(M) = 1$.)

\item[(ii)]
This example shows that
there is no analogue of
Proposition~\ref{prop:MP23}
for modules (instead of
components), compare however Theorem~\ref{thm:reduction1}.
Let $A = KQ$ where $Q$ is the quiver
$$
\xymatrix{
1 & \ar[l]_a 2 & 3 \ar[l]_b
}
$$
and let $B = KQ/(ab)$.
Furthermore, let $M = S(1) \oplus S(2) \oplus S(3)$.
Then $M$ is $\tau_A$-regular, but not $\tau_B$-regular.
(We have $c_A(M) = E_A(M) = 2$, $c_B(M) = 1$ and
$E_B(M) = 2$.)

\end{itemize}

For $M,N \in \md(A)$ let
$$
\trace_M(N) := \sum_{f \in \Hom_A(M,N)} \Ima(f)
$$
be the \emph{trace} of $M$ in $N$.

\begin{Thm} \label{thm:reduction1}
Let $e$ be an idempotent in $A$, and let $B := A/AeA$.
Then for $M \in \md(B)$ the following are equivalent:
\begin{itemize}\itemsep2mm

\item[(i)]
$M$ is $\tau_A$-regular;

\item[(ii)]
$M$ is $\tau_B$-regular.

\end{itemize}
\end{Thm}

\begin{proof}
Assume that $M \in \md(B)$.
Then we obviously have $c_A(M) = c_B(M)$.
Let
$$
P_1 \to P_0 \to M \to 0
$$
be a minimal projective presentation
in $\md(A)$.
The ideal $I := AeA$ is a $1$-idempotent ideal in the sense of
\cite[Section~1]{APT92}.
Then \cite[Theorem~1.6]{APT92} implies that the induced sequence
$$
P_1/IP_1 \to P_0/IP_0 \to M \to 0
$$
is a minimal projective presentation in $\md(B)$.
Now \cite[Proposition~4.2]{AR77} says that
$\tau_B(M) \cong \Hom_A(B,\tau_A(M))$.
Recall that for $X \in \md(A)$, the $A$-module
$\Hom_A(B,X)$ is isomorphic to the largest
submodule $U$ of $X$ such that $IU = 0$, compare
\cite[Proof of Lemma~4.1]{AR77}.
We get
$$
E_A(M) = \hom_A(M,\tau_A(M)) = \hom_A(M,\tau_B(M))
= \hom_B(M,\tau_B(M)) = E_B(M).
$$
Thus $M$ is $\tau_A$-regular if and only if $M$ is
$\tau_B$-regular.
\end{proof}

\begin{Cor} \label{cor:reduction2}
Let $e$ be an idempotent in $A$, and let $B := A/AeA$.
Then
$$
\irr(B) \cap \irr^\tau(A) = \irr^\tau(B).
$$
\end{Cor}

\begin{proof}
Let $\calZ \in \irr(B)$.
We have to show that $\calZ \in \irr^\tau(A)$ if and only
if $\calZ \in \irr^\tau(B)$.
This follows directly from Theorem~\ref{thm:reduction1}.
\end{proof}

Note that the inclusion $\subseteq$ in Corollary~\ref{cor:reduction2}
follows already from Proposition~\ref{prop:MP23}.

The following direct consequence of Corollary~\ref{cor:reduction2}
is quite useful.

\begin{Cor}\label{cor:reduction3}
Let $e$ be an idempotent in $A$, and let $B := A/AeA$.
If
$\irr^\tau(B) \not= \irr^{\tau^-}(B)$, then
$\irr^\tau(A) \not= \irr^{\tau^-}(A)$.
\end{Cor}


\section{Extensions and quotients of generically $\tau$-regular components}\label{sec:extensions}


\subsection{Generic extensions by simple projectives}
For $\calZ_1,\calZ_2 \in \irr(A)$
recall from
Section~\ref{subsec:genericextension}
that $\cE(\calZ_1,\calZ_2)$ denotes the generic extension
of $\calZ_1$ by $\calZ_2$.

Note that
for a simple module $S$ its orbit is just a point and forms
of course an irreducible component.
In fact, we have $S = \md(A,\dimv(S))$.

\begin{Prop} \label{prop:ext}
Let $S$ a simple projective $A$-module, and let
$\calZ \in \irr^\tau(A,\bd)$.
Let
$$
\varepsilon_S^+(\calZ) := \varepsilon(\calZ,S).
$$
Then $\varepsilon_S^+(\calZ) \in \irr^\tau(A,\bd+\dimv(S))$.
Furthermore, we have
$$
g(\varepsilon_S^+(\calZ)) =  g(\calZ) - [S].
$$
\end{Prop}

\begin{proof}
By Theorem~\ref{thm:CBS} we know that
$\calZ' := \varepsilon_S^+(\calZ)$ is an irreducible component of $\md(A,\bd + \dimv(S))$.

Let
\[
\calU := \{ M \in \calZ \mid \ext_A^1(M,S) = \ext^1(\calZ,S) \}.
\]
Let $\calU'$ be the set of all $M' \in \md(A,\bd + \dimv(S))$ such that
there exists a short exact sequence
$$
0 \to S \to M' \to M \to 0
$$
with $M \in \calU$.
Then $\calZ' = \overline{\calU'}$.

Since $\calZ$ is generically $\tau$-regular,
we have $\calZ = \calZ_{P_1,P_0}$ for some pair $(P_1,P_0)$
of projective $A$-modules, see Theorem~\ref{thm:plamondon1}.

Let $\cU_0 := \calZ_{P_1,P_0}^\circ$.
Then $\calZ = \overline{\cU_0}$.
By shrinking $\calU_0$, if necessary, we may assume that $\calU_0 \subseteq \calU$.
We may also assume that
for every
$f \in \Hom_A(P_1,P_0)^\circ$,
no non-zero direct summand of $P_1$ is contained in $\Ker(f)$,
compare Corollary~\ref{cor:bundle3}.

We claim that $r(P_1,P_0 \oplus S) = r(P_1,P_0)$.

Indeed, the inequality $r(P_1,P_0 \oplus S) \geq r(P_1,P_0)$ is obvious.
In order to prove the reverse inequality, fix a homomorphism
$$
h =
\left(\bbm
f \\ g
\ebm\right)
\colon P_1 \to P_0 \oplus S.
$$
Note that
$\rk(h) \leq \rk(f) + \dim(S) = \rk(f) + 1$.
Consequently, if $\rk(f )< r(P_1,P_0)$, then
$\rk(h) \leq r(P_1,P_0)$.
Similarly,
if $g = 0$, then
$\rk(h)=\rk(f)\leq r(P_1,P_0)$.
Thus we assume $\rk(f) = r(P_1,P_0)$ and $g \neq 0$.
It follows that $\Ker(f) \subseteq \Ker(g)$.
(Otherwise the composition $\Ker(f) \hookrightarrow P_1 \xrightarrow{g} S$ is a split epimorphism (since $S$ is simple projective), which gives a
non-zero direct summand of $P_1$ contained in $\Ker(f)$, a contradiction to our assumption.)
Since $\Ker(f) \subseteq \Ker(g)$, we get $\Ker(h) = \Ker(f) \cap \Ker(g) = \Ker(f)$.
This implies
$$
\rk(h) = \dim(P_1) - \dim \Ker(h)
= \dim(P_1) - \dim \Ker(f) = \rk(f) = r(P_1,P_0).
$$
This finishes the proof of the equality
$r(P_1,P_0 \oplus S) = r(P_1,P_0)$.

The above considerations also show that
$\Ker(h) = \Ker(f)$ and
$\rk(h) = r(P_1,P_0 \oplus S)$, provided
$\rk(f) = r(P_1,P_0)$.

Let $\calU_0'$ be the set of $M' \in \md(A,\bd + \dimv(S))$ such that there exists an exact sequence
$$
0 \to S \to M' \to M \to 0
$$
with $M \in \calU_0$.
Then $\calU_0'$ is a non-empty open subset of $\calU'$ by \cite[Theorem~1.3(iii)]{CBS02},
hence $\calZ' = \overline{\calU_0'}$.
Let $M' \in \calU_0'$, and fix an exact sequence
$$
0 \to S \to M' \to M \to 0
$$
with $M \in \calU_0$.
Since $S$ is projective,
the Horseshoe Lemma gives
a projective presentation
$$
P_1 \xrightarrow{h}
P_0 \oplus S \to M' \to 0
$$
where
$h = \left(\bbm
f \\ g
\ebm\right)
$
for some $g \in \Hom_A (P_1,S)$ and
$f \in \Hom_A(P_1,P_0)$ with $\rk(f)= r(P_1,P_0)$.
We get
$\rk(h) = r(P_1,P_0 \oplus S)$,
hence $M'$ is $\tau$-regular by Theorem~\ref{thm:main0}(i), thus
$\calZ'$ is
generically $\tau$-regular.
Moreover, $\Ker(h) = \Ker(f)$ together with Lemma~\ref{lem:gvector2} imply that
$g(\calZ') = g(\calZ) - [S]$.
\end{proof}

\subsection{Generic quotients by simple projectives}

\begin{Prop} \label{prop:quotient}
Let $S$ a simple projective $A$-module, and let
$\calZ \in \irr^\tau(A,\bd+\dimv(S))$.
Let
$$
\calZ' := \varepsilon_S^-(\calZ)
$$
be the closure of the set of all $M' \in \md(A,\bd)$ such that there exists a short exact sequence
$$
0 \to S \to M \to M' \to 0
$$
with $M \in \calZ$.
Then $\calZ' \in \irr^\tau(A,\bd)$.
Furthermore, we have
$$
g(\calZ') = g(\calZ) + [S].
$$
\end{Prop}

\begin{proof}
If follows from \cite[Section~2.1]{B94} that $\calZ'$ is irreducible.
(But note that we cannot yet conclude that $\calZ'$ is
an irreducible component of $\md(A,\bd)$.)
(Here we use that $S$ is projective, which implies
$\hom_A(S,M) = [M:S]$ for all $M \in \calZ$.)

Since $\calZ$ is generically $\tau$-regular,
there exists a pair $(P_1,P_0)$ of projective $A$-modules such that
$\calZ = \calZ_{P_1,P_0} :=  \overline{\calU}$ where
$\calU := \calZ_{P_1,P_0}^\circ$, see
Theorem~\ref{thm:plamondon1}.

We claim that $r(P_1 \oplus S,P_0) = r(P_1,P_0) + 1$.

Indeed, the inequality $r(P_1 \oplus S,P_0) \leq r(P_1,P_0) + 1$ is obvious, since $r(S,P_0) \leq 1$.
In order to prove the reverse inequality, choose
$f \in \Hom_A(P_1,P_0)^\circ$.
Thus
$\Coker(f) \in \calZ$.
In particular, we have
$\dimv(\Coker(f)) = \bd + \dimv(S)$.
Since $S$ is projective, there exists a non-zero map $g'\colon S \to \Coker(f)$.
Using the projectivity of $S$ again, we get a map $g\colon S \to P_0$ such that $\pi g = g'$,
where $\pi\colon P_0 \to \Coker(f)$ is the canonical projection.
$$
\SelectTips{cm}{}
\xymatrix{
&& S \ar@{-->}[dl]_g \ar[d]^<<<<{g'}
\\
P_1 \ar[r]^{f} & P_0 \ar[r]^<<<<<{\pi} & \Coker(f) \ar[r] & 0
}
$$
In particular, $\Ima(g) \not \subseteq \Ker(\pi) = \Ima(f)$.
Thus for $h := (f,g)\colon P_1 \oplus S \to P_0$ we get
$\rk(h) = \rk(f) + 1 = r(P_1,P_0) + 1$.
This finishes the proof of the equality $r(P_1 \oplus S,P_0) = r(P_1,P_0) + 1$.

Let $\calU' := \calZ_{P_1 \oplus S,P_0}^\circ$.
By Theorem~\ref{thm:plamondon1},
$\calZ'' := \cZ_{P_1 \oplus S,P_0} := \overline{\calU'}$
is a generically $\tau$-regular component of $\md(A,\bd)$.
To show that
$\calZ'$ is a generically $\tau$-regular component, we show that
$\calZ' = \calZ''$.
It is enough to show that $\calU' \subseteq \calZ'$. (Recall that $\calZ'$ is irreducible.)

Fix $h := (f,g)
\colon P_1 \oplus S \to P_0$ with $\rk(h) = r(P_1 \oplus S,P_0) =
r(P_1,P_0) + 1$, and let $M' := \Coker(h)$.
Observe that $M' \in \calU'$.
Since $\Ima(h) = \Ima(f) + \Ima(g)$, $\rk(f) \leq r(P_1,P_0)$, and
$\rk(g) \leq 1$, it follows that
$\rk(f) = r(P_1,P_0)$ and $\rk(g) = 1$.
Consequently, $M := \Coker(f) \in \calZ$.
Note that
$\Ima(f) \cap \Ima(g) = 0$, and therefore $\Ker(h) = \Ker(f)$.

We obtain a commutative diagram
$$
\SelectTips{cm}{}
\xymatrix{
& S \ar[d]^g
\\
P_1 \ar[r]^f\ar[d]_{\left(\bsm1\\0\esm\right)} & P_0 \ar[r]^{\pi}\ar[d]^1
& M \ar[r]\ar@{-->}[d]^p & 0
\\
P_1 \oplus S \ar[r]^<<<<{(f,g)} & P_0 \ar[r]^{\pi'} & M' \ar[r] & 0
}
$$
with exact rows,
where $\pi\colon P_0 \to M$ and $\pi'\df P_0 \to M'$ are the obvious
projections.

Now
an easy diagram chase
shows that
$$
0 \to S \xrightarrow{\pi g} M \xrightarrow{p} M' \to 0
$$
is a short exact sequence.
This shows that $M' \in \calZ'$.
Thus we proved that $\calU' \subseteq \calZ'$.

Finally, since $\Ker(h) = \Ker(f)$, we get
$g(\calZ') = g(\calZ) + [S]$
by Lemma~\ref{lem:gvector2}.
\end{proof}

We call the component $\calZ'$ constructed in Proposition~\ref{prop:quotient} the \emph{generic quotient} of $\calZ$ by $S$.

\subsection{Extensions and quotients by simple projectives
are mutually inverse}

\begin{Thm}\label{thm:main2b}
Let $S$ be a simple projective $A$-module.
For a dimension vector $\bd$ let $\bd' := \bd + \dimv(S)$.
Then the maps $\varepsilon_S^+$ and $\varepsilon_S^-$ yield
mutually inverse bijections
$$
\SelectTips{cm}{}
\xymatrix{
\irr^\tau(A,\bd) \ar@/^1ex/[rr]^{\varepsilon_S^+} &&
\irr^\tau(A,\bd')
\ar@/^1ex/[ll]^{\varepsilon_S^-}.
}
$$
\end{Thm}

\begin{proof}
Let $\calZ \in \irr^\tau(A,\bd)$.
It follows
from Propositions~\ref{prop:ext} and \ref{prop:quotient} that
$$
\calZ' := \varepsilon_S^-(\varepsilon_S^+(\calZ)) \in
\irr^\tau(A,\bd)
$$
and that $g(\calZ') = g(\calZ)$.
Now Theorem~\ref{thm:plamondon2} implies that $\calZ' = \calZ$.

Similarly, for $\calZ \in \irr^\tau(A,\bd')$ we have
$\calZ = \varepsilon_S^+(\varepsilon_S^-(\calZ))$.
\end{proof}

\subsection{Examples}

\subsubsection{}\label{ex:reduction2}

Let $A = KQ/I$ where $Q$ is the quiver
$$
\xymatrix{
1 & 2 \ar[l]_a & 3 \ar[l]_b
}
$$
and $I$ is generated by $ab$.
The Auslander-Reiten quiver of $A$ looks as follows:
$$
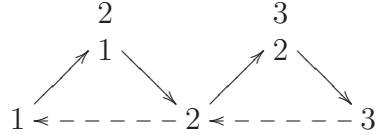
\xymatrix@!@-5ex{
& {\bbm 2\\1 \ebm} \ar[dr] && {\bbm 3\\2 \ebm} \ar[dr]
\\
1 \ar[ur] && 2 \ar[ur]\ar@{-->}[ll] && 3 \ar@{-->}[ll]
}
$$
Let
$$
M_1 = \bbm 2\\1 \ebm,\quad
M_2 = \bbm 3 \ebm,
\text{\quad and \quad}
M = M_1 \oplus M_2.
$$
For $i = 1,2$ let $\calZ_i := \overline{\cO_{M_i}}$.
Both are generically $\tau$-regular components, and
we have $\ext_A^1(\calZ_i,\calZ_j) = 0$ for $\{ i,j \} = \{ 1,2 \}$.
Then
$$
\calZ' := \cE(\calZ_1,\calZ_2) = \cE(\calZ_2,\calZ_1) = \overline{\cO_M}
$$
is an irreducible component, but it is not generically $\tau$-regular.
Thus a generic extension of a generically $\tau$-regular component by a non-projective simple does not necessarily give
a generically $\tau$-regular component.

\subsubsection{}
In the situation of Example~\ref{ex:reduction2},
the simple module
$S(1)$  is projective.
However, taking the generic quotient of $\calZ'$ by
$S(1)$ yields the closed irreducible subset
$\calZ := \overline{\cO_L}$ where
$$
L = \bbm 2 \ebm \oplus \bbm 3 \ebm.
$$
(In fact, $\overline{\cO_L} = \cO_L = L$ is just a point.)
Clearly, $\calZ$ is not an irreducible component.
Thus a generic quotient of an arbitrary irreducible component by a simple projective module does in general not give
an irreducible component.


\section{Generically $\tau$-regular components for
triangular algebras}\label{sec:triangular}


\subsection{Description of $\irr^\tau(A)$ for triangular algebras}
Let $A$ be a triangular algebra,
i.e.\ we assume that the simple
$A$-modules $S(1),\ldots,S(n)$ are labelled such that
$\Ext_A^1(S(i),S(j)) = 0$ for all $1 \le i \le j \le n$.

\begin{Thm}\label{thm:main3b}
Let $A$ be a triangular algebra as above.
Then for each dimension vector $\bd = (d_1,\ldots,d_n)$ there is a unique generically
$\tau$-regular
component $\calZ_\bd$ in $\irr(A,\bd)$, namely
$$
\calZ_\bd := (\varepsilon_{S(1)}^+)^{d_1} (\varepsilon_{S(2)}^+)^{d_2}  \cdots
 (\varepsilon_{S(n)}^+)^{d_n}(0).
$$
\end{Thm}

\begin{proof}
For $1 \le k \le n$ let
$$
e(k) := \sum_{i=1}^{k-1} e_i,
$$
and set $B(k) := A/Ae(k)A$.
Note that $e(1) = 0$ and therefore $B(1) = A$,
and that $B(n) \cong K$.

The module $S(k)$ is a simple projective $B(k)$-module.

First, we show that $\irr^\tau(A,\bd) \not= \varnothing$.

Let $\bd(k) := (0,\ldots,0,d_k,\ldots,d_n)$.

We obviously have
$$
\calZ(n) := (\varepsilon_{S(n)}^+)^{d_n}(0) = S(n)^{d_n} \in \irr^\tau(B(n),\bd(n)).
$$
(Note that $\md(B(n),\bd(n))$ is just a single point,
namely the projective $B(n)$-module $S(n)^{d_n}$.)

Let $1 \le k \le n-1$.
We have
$B(k+1) = B(k)/B(k)e_kB(k)$.

By decreasing induction on $k$ we can assume that
$$
\calZ(k+1) :=
(\varepsilon_{S(k+1)}^+)^{d_{k+1}} \cdots
(\varepsilon_{S(n)}^+)^{d_n}(0) \in
\irr^\tau(B(k+1),\bd(k+1)).
$$
Corollary~\ref{cor:reduction2} implies
$\calZ(k+1) \in \irr^\tau(B(k),\bd(k+1))$.

Now let
$$
\calZ(k) :=
(\varepsilon_{S(k)}^+)^{d_k}(\calZ(k+1)).
$$
Then
Proposition~\ref{prop:ext} tells us that
$\calZ(k) \in \irr^\tau(B(k),\bd(k))$.

In particular, we get
$\calZ_\bd = \calZ(1) \in \irr^\tau(A,\bd)$.

Next, we show uniqueness.
Let $\calZ_1, \calZ_2 \in \irr^\tau(A,\bd)$.
We can assume $\bd \not= 0$.
Let $k := \min\{ 1 \le i \le n \mid d_i \not= 0 \}$.

By Corollary~\ref{cor:reduction2} we get
$\calZ_1, \calZ_2 \in \irr^\tau(B(k),\bd)$.
Recall that $S(k)$ is a simple projective $B(k)$-module.

Let $\calZ_1' := \varepsilon_{S(k)}^-(\calZ_1)$ and
$\calZ_2' := \varepsilon_{S(k)}^-(\calZ_2)$ be the generic quotients of $\calZ_1$ and $\calZ_2$ by $S(k)$, respectively.
By  Proposition~\ref{prop:quotient} we get
$\calZ_1', \calZ_2' \in \irr^\tau(B(k),\bd - \dimv(S(k)))$.
By induction on $|\bd| := d_1 + \cdots + d_n$ we have
$\calZ_1' = \calZ_2'$.
If $\calZ := \varepsilon_{S(k)}^+(\calZ_1')$ is the generic extension of $\calZ_1' = \calZ_2'$ by $S(k)$, then Theorem~\ref{thm:main2b}
implies $\calZ = \calZ_1 = \calZ_2$.
\end{proof}

\subsection{Examples}\label{ex:reduction}

\subsubsection{}
Let $A = KQ/I$
where $Q$ is the quiver
$$
\xymatrix{
1 & 2 \ar@/^1ex/[l]^c\ar@/_1ex/_a[l] &
3 \ar@/^1ex/[l]^d\ar@/_1ex/_b[l]
}
$$
and $I = (ab,cd)$.
This is a triangular gentle algebra.
The irreducible components in $\irr(A)$ can be written as
$$
\calZ_{(r_a,r_b),(r_c,r_d)} =
\calZ_{(r_a,r_b)} \times \calZ_{(r_c,r_d)}
$$
where $(r_a,r_b)$ and $(r_c,r_d)$ are maximal pairs
in the sense of Example~(ii) in Section~\ref{subsec:irreducible}.
Let $\bd = (d_1,d_2,d_3) \in \N^3$.
Then
$$
\calZ_\bd = \calZ_{(r_a,r_b),(r_c,r_d)}
$$
where $r_b = r_d = \min\{ d_2,d_3 \}$ and
$r_a = r_c = \min \{ d_1,d_2-r_b \}$.
For example, for $\bd = (4,5,3)$ we have
$\calZ_\bd = \calZ_{(2,3),(2,3)}$.

This recipe can be easily extended to all triangular gentle algebras.

\subsubsection{}\label{ex:reduction1}
In the situation of Example~\ref{ex:reduction2},
for $\bd = (1,4,2)$ we get $\calZ_\bd = \calZ_{(1,2)} = \overline{\cO_M}$
where
$$
M = \bbm 2\\1 \ebm \oplus \bbm 2 \ebm \oplus
\bbm 3\\2 \ebm \oplus \bbm 3\\2 \ebm.
$$

\subsubsection{}\label{ex:reduction3}
The following easy example shows that there can be more than one
generically $\tau$-regular component with the same dimension vector.
Let $A = KQ/I$ where $Q$ is the quiver
$$
\xymatrix@-1ex{
1 \ar@/^1ex/[r]^a& 2 \ar@/^1ex/[l]^b
}
$$
and $I = (ab,ba)$.
Then
$$
P(1) = \bbm 1\\2 \ebm
\text{\quad and \quad}
P(2) = \bbm 2\\1 \ebm.
$$
For $i = 1,2$ let $\calZ_i = \overline{\cO_{P(i)}}$.
Both are generically $\tau$-regular components in
$\irr(A,(1,1))$.

\subsubsection{}\label{ex:reduction4}
Let $A = KQ/I$ where $Q$ is the quiver
$$
\xymatrix@-1ex{
1 \ar@(ur,dr)[]^a
}
$$
and $I$ is generated by $a^m$ for some $m \ge 2$.
For each $M \in \ind(A)$ one easily checks that
$\calZ := \overline{\cO_{M}} = \md(A,(d_1))$
where $d_1 = \dim(M)$.
In particular, $\calZ \in \irr(A)$.
Then $\calZ$
is generically $\tau$-regular if and only if $M$ is projective.
So for all dimension vectors $\bd = (d_1)$ with $d_1$
not divisible by $m$, there is no generically $\tau$-regular component in
$\irr(A,\bd)$.


\section{When all or very few components are generically $\tau$-regular}\label{sec:alltaureg}


\subsection{Algebras where all components are generically
$\tau$-regular}\label{subsec:alltaureg}
The idea of the following construction is this:
For an arbritrary $\calZ \in \irr(A)$ we pick a suitable
projective module.
Taking  the generic extension of $\calZ$ by this projective module yields a
new irreducible component $\calZ'$ whose $g$-vector has no positive entries.

Let $\calZ \in \irr(A,\bd)$, and let
$$
P_1 \xrightarrow{f} P_0 \xrightarrow{g} M \to 0
$$
be a minimal projective presentation where $M$ is generic in $\calZ$.
Thus we have $(P_1,P_0) = (P_1^{\calZ},P_0^{\calZ})$ and
$g(\calZ) = [P_1] - [P_0]$.
The sequence
$$
P_1 \xrightarrow{\left(\bsm f\\0\esm\right)}
P_0 \oplus P_1
\xrightarrow{\left(\bsm g&0\\0&1\esm\right)} M \oplus P_1 \to 0
$$
is then obviously a (minimal) projective presentation of
$M \oplus P_1$.

Let $\calZ_2 := \overline{\cO_{P_1}}$, and let
$\calZ' := \cE(\calZ,\calZ_2)$ be the generic extension of
$\calZ$ by $\calZ_2$, as defined in Section~\ref{subsec:genericextension}.
Since $\ext_A^1(\calZ_2,\calZ) = 0$, Theorem~\ref{thm:CBS}
says that
$\calZ' \in \irr(A,\bd')$ where
$\bd' := \bd + \dimv(P_1)$.

Since $M \oplus P_1 \in \calZ'$, we get
$$
g(\calZ') = [P_1] - [P_0 \oplus P_1] - [P] = -[P_0] - [P]
$$
for some direct summand $P$ of $P_1$.
(Here we used Lemmas~\ref{lem:gvector1} and  \ref{lem:presentationsemicont}.)

\begin{Lem}\label{lem:alltaureg}
With $\calZ$ and $\calZ'$ as above, the following are equivalent:
\begin{itemize}\itemsep2mm

\item[(i)]
$\pdim(\calZ) \le 1$;

\item[(ii)]
$\calZ' = \overline{\cO_{P_0}}$;

\item[(iii)]
$\calZ' \in \irr^\tau(A)$.

\end{itemize}
\end{Lem}

\begin{proof}
(i) $\implies$ (ii):
If $\pdim(\calZ) \le 1$, then
there is an exact sequence
$$
0 \to P_1 \to P_0 \to M \to 0
$$
for some $M \in \calZ$ with $\ext_A^1(M,\calZ_2) = \ext_A^1(\calZ,\calZ_2)$.
It follows that $P_0 \in \calZ'$ and therefore
$$
\calZ' = \overline{\cO_{P_0}}.
$$

(ii) $\implies$ (iii): Trivial.

(iii) $\implies$ (i):
If $\calZ' \in \irr^\tau(A)$, then
$$
\calZ' = \overline{\cO_{P_0 \oplus P}},
$$
since $g(\calZ') = -[P_0]-[P]$ and since generically $\tau$-regular components are determined by their $g$-vectors, see Theorem~\ref{thm:plamondon2}.
Furthermore, there is a short exact sequence
$$
0 \to P_1 \to P_0 \oplus P \to M \to 0
$$
for some $M \in \calZ$.
Thus $\pdim(M) \le 1$ and therefore $\pdim(\calZ) \le 1$.
\end{proof}

Note that the module $P$ occuring in the proof of
Lemma~\ref{lem:alltaureg} turns out to be $0$.

\begin{Thm}\label{thm:main1b}
The following are equivalent:
\begin{itemize}\itemsep2mm

\item[(i)]
$\irr^\tau(A) = \irr(A)$;

\item[(ii)]
$\irr^{\tau^-}(A) = \irr(A)$;

\item[(iii)]
$A$ is hereditary.

\end{itemize}
\end{Thm}

\begin{proof}
(iii) $\implies$ (i),(ii):
If $A$ is hereditary, then $\md(A,\bd)$ is an affine space for all
$\bd$.
Thus $\dim \md_M(A,\bd) = \dim(T_M)$ for all
$M \in \md(A,\bd)$.
Now Voigt's Lemma \ref{prop:voigt},
Theorem~\ref{thm:ARformulas} and
Lemma~\ref{lem:pdim1stable}
imply (i).
Using the dual of Lemma~\ref{lem:pdim1stable}, one shows that (ii) also holds.

(i) $\implies$ (iii):
Let $1 \le i \le n$, and let
$\calZ := \cO_{S(i)}$.
This is an irreducible component consisting of a single point,
namely $S(i)$.
Let $\calZ'$ be the associated component appearing in Lemma~\ref{lem:alltaureg}.
By (i) the component $\calZ'$ must be generically $\tau$-regular.
Now Lemma~\ref{lem:alltaureg} says that $\pdim(\calZ) \le 1$.
Thus $\pdim(S(i)) \le 1$ for all $i$.
This implies that $\gldim(A) \le 1$, i.e.\ (iii) holds.

(ii) $\implies$ (iii):
This is done similarly by using the dual of
Lemma~\ref{lem:alltaureg}.
\end{proof}

\subsection{Algebras where very few components are generically
$\tau$-regular}\label{subsec:fewtaureg}
We have seen that hereditary algebras are the only algebras
where $\irr^\tau(A)$ is as big as possible, namely
$\irr^\tau(A) = \irr(A)$.
The next result says that local algebras are the only algebras where
$\irr^\tau(A)$ is as small as possible.

\begin{Prop}\label{prop:local}
The following are equivalent:
\begin{itemize}\itemsep2mm

\item[(i)]
$A \cong A_1 \times \cdots \times A_t$ where all
$A_i$ are local algebras;

\item[(ii)]
$\irr^\tau(A) = \{ \overline{\cO_P} \mid P \in \proj(A) \}$;

\item[(iii)]
$\proj(A) = \text{$\tau$-$\rigid(A)$}$;

\item[(iv)]
$\irr^{\tau^-}(A) = \{ \overline{\cO_I} \mid I \in \inj(A) \}$;

\item[(v)]
$\inj(A) = \text{$\tau^-$-$\rigid(A)$}$.

\end{itemize}
\end{Prop}

\begin{proof}
We assume without loss of generality that $A$ is connected.
In particular,
in (i) we assume that $t=1$.

We clearly have $\proj(A) \subseteq \text{$\tau$-$\rigid(A)$}$
and
$\inj(A) \subseteq  \text{$\tau^-$-$\rigid(A)$}$.

(i) $\implies$ (ii):
If $A$ is local, then there exists just one indecomposable projective
$P(1)$, up to isomorphism.
We have
$g(P(1)^m) = -m$ for $m \ge 0$.
Let $\calZ_m := \overline{\cO_{P(1)^m}}$.
Obviously, $\calZ_m \in \irr^\tau(A)$.
We get
$$
g^\circ(\calZ_m) =
\begin{cases}
-m & \text{if $m \ge 1$},
\\
\N & \text{if $m = 0$}.
\end{cases}
$$
Thus
$$
\Z = \bigcup_{m \ge 0} g^\circ(\calZ_m).
$$
Now Theorem~\ref{thm:plamondon2} implies (ii).

(ii) $\implies$ (iii):
This is clear, since for each $\tau$-rigid module $M$ we have
$\overline{\cO_M} \in \irr^\tau(A)$.

(iii) $\implies$ (i):
Assume that $A$ is not local.
Since $A$ is connected, there exists some $P(i)$ and some
$j \not= i$ such that $[P(i):S(j)] \not= 0$.
Let $E(i)$ be the largest factor module of $P(i)$ such that
$[E(i):S(k)] = 0$ for all $k \not= i$.
The module $E(i)$ is non-zero and non-projective.
We get a minimal projective presentation
$$
P_1 \to P(i) \to E(i) \to 0
$$
such that $P_1$ does not have a direct summand isomorphic to
$P(i)$.
This yields a minimal injective presentation
$$
0 \to \tau_A(E(i)) \to \nu_A(P_1) \to \nu_A(P(i)).
$$
In particular, we have
$$
\soc(\tau_A(E(i))) \cong \soc(\nu_A(P_1)) \cong \tp(P_1).
$$
Thus $\tau_A(E(i))$ does
not have a submodule isomorphic to $S(i)$.
In other words, $\Hom_A(E(i),\tau_A(E(i)) = 0$.
So $E(i)$ is a non-projective $\tau$-rigid module.

The equivalence of (i), (iv) and (v) is proved dually.
\end{proof}


\section{Generically $\tau$-regular versus generically $\tau^-$-regular components}
\label{sec:tauversustauminus}


\subsection{Correspondence of $\tau$-regular pairs and
$\tau^-$-regular pairs}\label{subsec:bijectionpairs}

Recall from \cite{AIR14} that a pair $(M,P)$ of $A$-modules
is a
\emph{$\tau$-rigid pair} if $M$ is $\tau$-rigid and $P$ is projective
with $\Hom_A(P,M) = 0$.
Dually, one defines \emph{$\tau^-$-rigid pairs}.
We consider such pairs up to isomorphism.

Let $\tau$-rigid-pairs$(A)$ (resp. $\tau^-$-rigid-pairs$(A)$) be the
set (of isomorphism classes) of $\tau$-rigid pairs (resp. $\tau^-$-rigid pairs) of $A$-modules.

For $M \in \md(A)$ let $P_M$ be a maximal projective direct summand
of $M$.

Applying the duality $D$ to the first bijection in \cite[Theorem~2.14]{AIR14} one gets the following theorem:

\begin{Thm}[{Adachi, Iyama, Reiten}]
The map
$$
(M,P) \mapsto (\tau(M) \oplus \nu_A(P),\nu_A(P_M))
$$
defines a bijection
$$
\tau\text{\rm -rigid-pairs}(A) \to \tau^-\text{\rm -rigid-pairs}(A).
$$
\end{Thm}

Recall that for $\calZ \in \irr^\tau(A)$ and
$P \in \proj(A)$ we call $(\calZ,P)$ a \emph{$\tau$-regular pair}
provided $\Hom_A(P,M) = 0$ for all $M \in \calZ$.
Dually, one defines \emph{$\tau^-$-regular pairs}.

Let $\tau$-reg-pairs$(A)$ (resp. $\tau^-$-reg-pairs$(A)$) be the
set (of isomorphism classes) of $\tau$-regular pairs (resp. $\tau^-$-regular pairs) for $A$.

For $\calZ \in \irr^\tau(A)$
let $\tau(\calZ)$ be the generically $\tau^-$-regular component
$\calZ'$ such that
$$
(I_0^{\calZ'},I_1^{\calZ'}) = (\nu_A(P_1^{\calZ}),\nu_A(P_0^{\calZ})).
$$
(Here we are using the notation from Lemmas~\ref{lem:genericpresentations1} and \ref{lem:genericpresentations2}.)
Dually, one defines $\tau^-(\calZ)$ for $\calZ \in \irr^{\tau^-}(A)$.

Let  $P_{\calZ}$ (resp. $I_\calZ$) be a maximal projective (resp. injective) direct summand of
some generic $M \in \calZ$.

\begin{Prop}\label{prop:correspondence}
The map
\begin{align*}
\sigma\df \tau\text{\rm-reg-pairs}(A) &\to
\tau^-\text{\rm-reg-pairs}(A)
\\
(\calZ,P) &\mapsto (\overline{\tau(\calZ) \oplus \cO_{\nu_A(P)}},\nu_A(P_{\calZ}))
\end{align*}
is a bijection.
\end{Prop}

\begin{proof}
Let $(\calZ,P)\in \tau\text{\rm-reg-pairs}(A)$.
First, we show that
$(\overline{\tau(\calZ) \oplus \cO_{\nu_A(P)}},\nu_A(P_{\calZ}))$
is a $\tau^-$-regular pair for $A$.

By construction, we have $\tau(\calZ) \in \irr^{\tau^-}(A)$.

Each generic $M \in \calZ$ is of the form
$M_{\rm np} \oplus P_\calZ$ for some $M_{\rm np}$ which does
not have a non-zero projective direct summand.
By \cite[Theorem~1.3]{CLS15} we have
$\Hom_A(P_\calZ,\tau_A(M_{\rm np})) = 0$.
By construction,
a generic module in $\tau(\calZ)$ is of the form
$\tau_A(M_{\rm np})$.
(This was already observed in \cite[Theorem~3.11(2)]{F23}.)

We have
$$
\Hom_A(\tau_A^-(\tau_A(M_{\rm np})),\nu_A(P))
\cong \Hom_A(M_{\rm np},\nu_A(P)) \cong \Hom_A(P,M_{\rm np}) = 0
$$
and
$$
\Hom_A(\tau_A^-(\nu_A(P)),\tau_A(M_{\rm np})) = 0.
$$
Now the dual of \cite[Theorem~1.3]{CLS15} says that
$$
\calZ' :=
\overline{\tau(\calZ) \oplus \cO_{\nu_A(P)}} \in \irr^{\tau^-}(A).
$$
Furthermore, we have
\begin{align*}
\Hom_A(\tau_A(M_{\rm np}) \oplus \nu_A(P),\nu_A(P_\calZ))
&\cong \Hom_A(P_\calZ,\tau_A(M_{\rm np}) \oplus \nu_A(P))
\\
&\cong  \Hom_A(P_\calZ,\tau_A(M_{\rm np})) \oplus \Hom_A(P,P_\calZ) = 0.
\end{align*}
Thus $\Hom_A(M',\nu_A(P_\calZ)) = 0$ for all $M' \in \calZ'$.
It follows that
$(\calZ',\nu_A(P_\calZ))$ is
a $\tau^-$-regular pair for $A$.

For a $\tau^-$-regular pair $(\calZ,I)$ define
$$
\sigma^{-1}(\calZ,I) :=
(\overline{\tau^-(\calZ) \oplus \cO_{\nu_A^-(I)}},\nu_A^-(I_\calZ)).
$$
Now one checks easily that $\sigma^{-1}$ is the inverse of
$\sigma$.
\end{proof}

\subsection{When do $\tau$-regular and $\tau^-$-regular components coincide?}
Let $m \ge 0$.
Then $A$ is an
\emph{$m$-Iwanaga-Gorenstein algebra} if
$$
\pdim(D(A_A)) \le m
\text{\quad and \quad}
\idim({_A}A) \le m.
$$
In this case, one has
$d := \pdim(D(A_A)) = \idim({_A}A)$.
Furthermore, for all $M \in \md(A)$ the following are equivalent:
\begin{itemize}\itemsep2mm
\item $\pdim(M) \le d$;
\item $\pdim(M) < \infty$;
\item $\idim(M) \le d$;
\item $\idim(M) < \infty$.
\end{itemize}

\begin{Thm}\label{thm:main4b}
Suppose that $\irr^\tau(A) = \irr^{\tau^-}(A)$.
Then $A/AeA$ is a $1$-Iwanaga-Gorenstein algebra for each
idempotent $e$ in $A$.
\end{Thm}

\begin{proof}
Under the assumption we get that
each $I \in \inj(A)$ is $\tau$-rigid, and each
$P \in \proj(A)$ is $\tau^-$-rigid.
Now Lemma~\ref{lem:pdim1} and its dual
imply that $\pdim(I) \le 1$ and $\idim(P) \le 1$, respectively.

Next, let $e$ be an idempotent in $A$, and let $B := A/AeA$.
Then $B$ is also a $1$-Iwanaga-Gorenstein algebra.
Otherwise
we get $\irr^\tau(B) \not= \irr^{\tau^-}(B)$, and then Corollary~\ref{cor:reduction2} and its dual imply
$\irr^\tau(A) \not= \irr^{\tau^-}(A)$, a contradition.
\end{proof}

In general,
it remains unclear when $\irr^\tau(A) = \irr^{\tau^-}(A)$ holds.
The following section
answers this question for several
classes of algebras.

\subsection{Examples}

\subsubsection{A selfinjective algebra with $\irr^\tau(A) \not= \irr^{\tau^-}(A)$}
Let $A = KQ/I$ where $Q$ is the quiver
$$
\xymatrix{
1 \ar@/^1ex/[r]^a & 2 \ar@/^1ex/[l]^b
}
$$
and $I = (ab,ba)$.
Observe that $A$ is a selfinjective algebra,
i.e.\ $A$ is $0$-Iwanaga-Gorenstein (and therefore also
$1$-Iwanaga-Gorenstein).
The $A$-module
$P(1) \oplus S(1)$
is $\tau$-rigid, but not $\tau^-$-rigid.
In particular, $\irr^\tau(A) \not= \irr^{\tau^-}(A)$.

\subsubsection{Local algebras}

\begin{Prop}
Assume $A$ is local.
Then the following are equivalent:
\begin{itemize}\itemsep2mm

\item[(i)]
$\irr^\tau(A) = \irr^{\tau^-}(A)$;

\item[(ii)]
$A$ is selfinjective.

\end{itemize}
\end{Prop}

\begin{proof}
This follows directly from Proposition~\ref{prop:local}.
\end{proof}

\subsubsection{Algebras of finite global dimension}

\begin{Prop}\label{prop:finiteglobaldim}
Assume that $\gldim(A) < \infty$.
Then the following are equivalent:
\begin{itemize}\itemsep2mm

\item[(i)]
$\irr^\tau(A) = \irr^{\tau^-}(A)$;

\item[(ii)]
$A$ is hereditary.

\end{itemize}
\end{Prop}

\begin{proof}
(i) $\implies$ (ii):
We know from Theorem~\ref{thm:main4b} that (i) implies that
$A$ is $1$-Iwanaga-Gorenstein.
But a $1$-Iwanaga-Gorenstein algebra of finite global dimension
is hereditary.

(ii) $\implies$ (i):
This follows from (the obvious part of) Theorem~\ref{thm:main1b}.
\end{proof}

\subsubsection{Generalized species of Dynkin type}

For a symmetrizable generalized Cartan matrix $C$,
a symmetrizer $D$ and an acyclic orientation $\Omega$,
one can define a $1$-Iwanaga-Gorenstein algebra
$A = H(C,D,\Omega)$ which is an analogue of a species, see
\cite{GLS17} for details.
The locally free $A$-modules (as defined in \cite{GLS17}) are exactly
the $A$-modules with projective dimension at most one.

\begin{Prop}
For $A = H(C,D,\Omega)$ we have
$\irr^\tau(A) = \irr^{\tau^-}(A)$.
\end{Prop}

\begin{proof}
Let $A = H(C,D,\Omega)$.
It follows from \cite[Theorem~1.2]{GLS17} that for $\calZ \in \irr(A)$ the following are equivalent:
\begin{itemize}\itemsep2mm

\item[(i)]
$\calZ$ is generically locally free;

\item[(ii)]
$\pdim(\calZ) \le 1$;

\item[(iii)]
$\idim(\calZ) \le 1$.

\end{itemize}
Pfeifer
\cite[Theorem~1.1]{Pf23}
shows that $\calZ \in \irr(A)$ is generically $\tau$-regular
if and only if $\calZ$ is generically locally free.
A dual proof yields the same for generically $\tau^-$-regular components.
\end{proof}

\subsubsection{Jacobian algebras}
For a $2$-acyclic quiver $Q$ and a potential $W$ for
$Q$ let $A = J(Q,W)$ be the associated Jacobian algebra introduced in
\cite{DWZ08}.
These algebras play a crucial role in the categorification of
Fomin-Zelevinsky cluster algebras.
(Note that the algebras $J(Q,W)$ can be infinite-dimensional, and for the
mentioned categorification one needs the additional assumption that
the potential $W$ is non-degenerate.)

For example, all gentle algebras arising from surface triangulations
(see \cite{ABCP10})
are of the form $J(Q,W)$.
We call these \emph{gentle surface algebras}.

Combining
\cite[Proposition~7.3]{DWZ10} and
\cite[Corollary~10.9]{DWZ10} yields the following
theorem.
Note that \cite[Proposition~7.3]{DWZ10}
is stated in \cite[Section~7]{DWZ10} where the potential $W$ is assumed to be non-degenerate for the whole section.
However the proposition clearly holds also under the assumptions
of \cite[Section~10]{DWZ10}.

\begin{Thm}\label{thm:jacobian}
Assume that $Q$ is a $2$-acyclic quiver, and let $W$ be a potential
for $Q$ such that
$W \in KQ$, and the ideal $I$ generated by the cyclic derivatives of $W$
is an admissible ideal in $KQ$.
Then for $A = J(Q,W) = KQ/I$ and each
$M \in \md(A)$ we have
$$
\hom_A(M,\tau_A(M)) = \hom_A(\tau_A^-(M),M).
$$
In particular, we get
$$
\irr^\tau(A) = \irr^{\tau^-}(A).
$$
\end{Thm}

\subsubsection{Gentle algebras}
Our finite-dimensional algebra $A = KQ/I$ is a
\emph{gentle algebra} if the following
hold:
\begin{itemize}\itemsep2mm

\item[(i)]
For each $i \in Q_0$ we have
$$
|\{ a \in Q_1 \mid s(a) = i \}| \le 2
\text{\quad and \quad}
|\{ a \in Q_1 \mid t(a) = i \}| \le 2;
$$

\item[(ii)]
$I$ is generated by a set of paths of length two;

\item[(iii)]
Let $a,b,c \in Q_1$ with $t(a) = s(b) = s(c)$ and $b \not= c$.
Then $|\{ ba,\; ca \} \cap I| = 1$;

\item[(iv)]
Let $a,b,c \in Q_1$ with $s(a) = t(b) = t(c)$ and $b \not= c$.
Then $|\{ ab,\; ac \} \cap I| = 1$.

\end{itemize}

These algebras first appeared in \cite{AS87}.

\begin{Prop}[{Chen, Lu \cite[Proposition~3.1]{CL16}}] \label{prop:CL}
A gentle algebra $A = KQ/I$ is $1$-Iwanaga-Gorenstein
if and only if for each path $a_ta_1$ of length two with
$a_ta_1 \in I$ there is a
path $a_1 \cdots a_t$ in $Q$ such that
$a_ia_{i+1} \in I$ for all $1 \le i \le t-1$.
\end{Prop}

Note that the path $a_1 \cdots a_t$ appearing in
Proposition~\ref{prop:CL} is an oriented cycle.

The proof of the following result is based on our reduction technique
from Corollary~\ref{cor:reduction2}.

\begin{Prop}\label{prop:gentle}
Let $A = KQ/I$ be a gentle algebra such that $Q$ has no loops.
Then the following are equivalent:
\begin{itemize}\itemsep2mm

\item[(i)]
$\irr^\tau(A) = \irr^{\tau^-}(A)$;

\item[(ii)]
$A$ is a gentle surface algebra.

\end{itemize}
\end{Prop}

\begin{proof}
(ii) $\implies$ (i):
All gentle surface algebras are Jacobian algebras $J(Q,W)$ and
satisfy the assumptions of
Theorem~\ref{thm:jacobian}.
This implies (i).

(i) $\implies$ (ii):
Since we asssume (i), we know from Theorem~\ref{thm:main4b} that $A/AeA$ is
$1$-Iwanaga-Gorenstein for each idempotent $e$ in $A$.

First, we assume that there is a $2$-cycle in $Q$.
Since $Q$ has no loops, there is an idempotent $e \in A$
such that $B = A/AeA \cong KQ'/I'$ is a gentle
$1$-Iwanaga-Gorenstein algebra
where $Q'$ is one of the following quivers:
$$
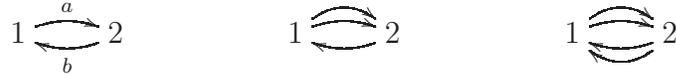
\xymatrix{
1 \ar@/^1ex/[r]^a& 2 \ar@/^1ex/[l]^b
&&
1 \ar@/^1ex/[r]\ar@/^2ex/[r] & 2 \ar@/^1ex/[l]
&&
1 \ar@/^1ex/[r]\ar@/^2ex/[r]& 2 \ar@/^1ex/[l]\ar@/^2ex/[l]
}
$$
For the first quiver, only the ideal $(ab,ba)$ gives a gentle $1$-Iwanaga-Gorenstein algebra.
But in this case, we know already that
$\irr^\tau(B) \not= \irr^{\tau^-}(B)$.
For the second and the third quiver, there is no ideal which yields a
gentle $1$-Iwanaga-Gorenstein algebra.

Thus we can assume that $Q$ does not have $2$-cycles.
If $A$ is hereditary, we are done.
Thus assume that $A$ is not hereditary.
Then there exists a $B = A/AeA \cong KQ'/I'$ where $Q'$ has
$3$ vertices, and $I'$ contains at least one  path
of length two, say $ba$.
Now, if $Q'$ is acyclic, then $B$ cannot be $1$-Iwanaga-Gorenstein,
since $\pdim(S_{s(a)}) = 2$.

Thus, it remains to consider the following four quivers:
$$
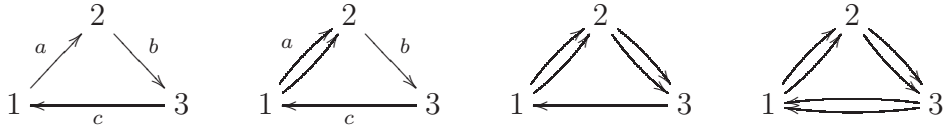
\xymatrix@-1ex{
& 2 \ar[dr]^b
&&
& 2 \ar[dr]^b
&&
& 2 \ar@/^0.5ex/[dr]\ar@/_0.5ex/[dr]
&&
& 2 \ar@/^0.5ex/[dr]\ar@/_0.5ex/[dr]
\\
1 \ar[ur]^a && 3 \ar[ll]^c
&
1 \ar@/^0.5ex/[ur]^a\ar@/_0.5ex/[ur] && 3 \ar[ll]^c
&
1 \ar@/^0.5ex/[ur]\ar@/_0.5ex/[ur] && 3 \ar[ll]
&
1 \ar@/^0.5ex/[ur]\ar@/_0.5ex/[ur] &&
3 \ar@/^0.5ex/[ll]\ar@/_0.5ex/[ll]
}
$$
In the first two cases, the ideal $(ba,cb,ac)$ is the only one which
yields a gentle $1$-Iwanaga-Gorenstein algebra.
Both are gentle surface algebras.
For the third and the fourth quivers there is no ideal
which yields a gentle $1$-Iwanaga-Gorenstein algebra.

This finishes the proof.
\end{proof}

If one considers gentle algebras $A = KQ/I$ where $Q$ has loops, then there are more cases to consider.
We leave this to the keen reader.

\subsubsection{Nakayama algebras}
Let $A = KQ/I$ where $Q$ is the cyclic quiver
$$
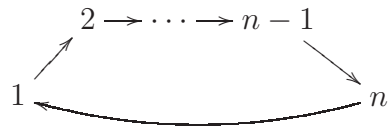
\xymatrix@-2ex{
& 2 \ar[r] & \cdots \ar[r] & n-1 \ar[dr]
\\
1 \ar[ur] &&&& n \ar@/^2ex/[llll]
}
$$
with $n$ vertices,
and $I$ is generated by some (non-empty) set of paths of length at least two.
In other words,
$A$ is a cyclic Nakayama algebra.
Since we dealt with local algebras already, we assume $n \ge 2$.
To simplify the notation, we work with indices modulo $n$.

The following three lemmas are well known and quite easy to prove.

\begin{Lem}
Let $A = KQ/I$ be a cyclic Nakayama algebra as above.
For each $1 \le i \le n$ we have
$$
\tau_A(S(i)) \cong S(i+1)
\text{\quad and \quad }
\tau_A^{-1}(S(i)) \cong S(i-1).
$$
\end{Lem}

\begin{Lem}\label{lem:selfinjnakayama1}
Let $A$ be a cyclic Nakayama algebra.
If $A$ is $1$-Iwanaga-Gorenstein, then $A$ is selfinjective.
\end{Lem}

Let $J$ be the ideal in $KQ$ generated by all arrows.

\begin{Lem}\label{lem:selfinjnakayama2}
A cyclic Nakayama algebra $A = KQ/I$ is selfinjective if and only
if there exists some $t \ge 2$ such that $I = J^t$.
\end{Lem}

\begin{Lem}[{Adachi \cite[Proposition~2.5]{A16}}]
\label{lem:taurigidnakayama}
Let $A = KQ/I$ be a cyclic Nakayama algebras as above,
and let $M \in \ind(A)$, and assume that $M$ is non-projective.
Then $M$ is $\tau$-rigid if and only if $\dim(M) \le n-1$.
\end{Lem}

\begin{Prop}\label{prop:nakayama}
Let $A = KQ/I$ be a cyclic Nakayama algebra as above.
Then the following are equivalent:
\begin{itemize}\itemsep2mm

\item[(i)]
$\irr^\tau(A) = \irr^{\tau^-}(A)$;

\item[(ii)]
$I = J^{(n-1) + nr}$ for some $r \ge 0$ (with $r \ge 1$ in case
$n=2$).

\end{itemize}
\end{Prop}

\begin{proof}
(i) $\implies$ (ii):
Under the assumption (i), combining
Theorem~\ref{thm:main4b} and Lemma~\ref{lem:selfinjnakayama1}
shows that
$A$ is selfinjective.
Thus by Lemma~\ref{lem:selfinjnakayama2} we have
$I = J^t$ for some $t \ge 2$.

First, let $n = 2$.
If $t$ is even, then $\soc(P(i)) \cong S(j)$ for $\{ i,j \} = \{ 1,2 \}$.
Thus $P(1) \oplus S(1)$ is $\tau$-rigid but not $\tau^-$-rigid.
Thus $t$ has to be odd, i.e.\ (ii) holds.

Next, let $n \ge 3$.
For $1 \le i,j \le n$ let $M_{ij} := S(i) \oplus P(j)$.
Then $M_{ij}$ is $\tau$-rigid if and only if $j \not= i+1$.

Let now $1 \le j \le n$.
We choose $i$ such that $\soc(P(j)) \cong S(i-1)$.
This implies $\Hom_A(\tau_A^-(M_{ij}),M_{ij}) \not= 0$.
Thus, if $j \not= i+1$, then $\irr^\tau(A) \not= \irr^{\tau^-}(A)$.

Assume now that $\irr^\tau(A) = \irr^{\tau^-}(A)$.
By the above, this implies
$\soc(P(j)) \cong S(j-2)$ for each $1 \le j \le n$.
It follows that there exists some $r \ge 0$ such that
$\dim(P(j)) = (n-1) +nr$ for all $j$.
This implies $I = J^{(n-1) + nr}$.

(ii) $\implies$ (i):
First, let $n=2$.
Then (ii) amounts to assuming that $I = J^t$ with $t$ odd.
By Lemma~\ref{lem:taurigidnakayama} the only indecomposable
$\tau$-rigid (resp. $\tau^-$-rigid) $A$-modules are $S(1),S(2),P(1),P(2)$ (resp. $S(1),S(2),I(1),I(2))$.
We know that $\tau_A(S(i)) \cong S(j) \cong \tau_A^-(S(i))$
for $\{ i,j \} = \{ 1,2 \}$.
Since $t$ is odd, we have $\soc(P(i)) \cong S(i)$ for $i=1,2$.
The $\tau$-rigid $A$-modules are
then $P(1)^{a_1} \oplus P(2)^{a_2}$,  $P(1)^{a_1} \oplus S(1)^{a_2}$,
 $P(2)^{a_1} \oplus S(2)^{a_2}$ for $a_1,a_2 \ge 0$, and these are also the
$\tau^-$-rigid modules,
i.e.\ we have
$\irr^\tau(A) = \irr^{\tau^-}(A)$.

Next,
assume that $n \ge 3$.
We have
$A = KQ/J^{(n-1) + nr}$ for some
$r \ge 0$.
Let $S$ be a cycle of length $n$ in $Q$, and set $W := S^{r+1}$.
This is a potential in the sense of \cite{DWZ08}.
Assume that ${\rm char}(K) \not= r+1$.
Then $J^{(n-1)+nr}$ equals the
ideal generated by the cyclic derivatives of
$W$.  (If ${\rm char}(K) = r+1$, then the cyclic derivatives of $W$
are $0$.)
Thus $J(Q,W)$ satisfies the assumptions of
Theorem~\ref{thm:jacobian}.
This implies
$\irr^\tau(A) = \irr^{\tau^-}(A)$.

Since $A$ is a Nakayama algebra, the classification of
indecomposable
$A$-modules and their Auslander-Reiten translates
and also $\hom_A(M,N)$ for all $M,N \in \md(A)$ do
not depend on the chosen ground field $K$.
This implies that the condition $\irr^\tau(A) = \irr^{\tau^-}(A)$
does not depend on $K$.

This finishes the proof.
\end{proof}

\bigskip
\subsection*{Acknowledgements}
GB gratefully acknowledges the support of the National Science Centre grant no.\ 2020/37/B/ST1/00127.
JS was partially funded by the Deutsche Forschungsgemeinschaft
(DFG, German Research Foundation) under Germany's Excellence Strategy - GZ 2047/1, Projekt-ID 390685813.
He also thanks the
Faculty of Mathematics and Computer Science
of the Nicolaus Copernicus University in Toru\'n for 10 days of hospitality in August 2024.


\end{document}